\documentclass{amsart}
\usepackage{color}

\usepackage{amsmath} 
\usepackage{amssymb}
\usepackage{mathrsfs}

\newtheorem{theorem}{Theorem}[section] 
\newtheorem{claim}[theorem]{Claim}
\newtheorem{lemma}[theorem]{Lemma} 
 
\newtheorem{observation}[theorem]{Observation} 
 
\newtheorem{conclusion}[theorem]{Conclusion}

\theoremstyle{definition}
\newtheorem{definition}[theorem]{Definition}
\newtheorem{example}[theorem]{Example}

\newtheorem{thesis}[theorem]{Thesis}
\newtheorem{conjecture}[theorem]{Conjecture}

\newtheorem{discussion}[theorem]{Discussion}

\theoremstyle{remark}
\newtheorem{remark}[theorem]{Remark}

\newtheorem{question}[theorem]{Question}
\newtheorem{context}[theorem]{Context}

\newcommand{\rest}{{\restriction}}

\newcommand{\Range}{{\rm Range}}
\newcommand{\Min}{{\rm Min}}
\newcommand{\Dom}{{\rm Dom}}

\newcommand{\Ord}{{\rm Ord}}
\newcommand{\proj}{{\rm proj}}
\newcommand{\rk}{{\rm rk}}
\newcommand{\tp}{{\rm tp}}

\newcommand{\wilog}{{\rm without loss of generality}}
\newcommand{\Wilog}{{\rm Without loss of generality}}

\newcommand{\then}{{\underline{then}}}
\newcommand{\when}{{\underline{when}}}
\newcommand{\where}{{\underline{where}}}
\newcommand{\Then}{{\underline{Then}}}

\newcommand{\Iff}{{\underline{iff}}}
\newcommand{\mn}{{\medskip\noindent}}
\newcommand{\sn}{{\smallskip\noindent}}

\newcommand{\cB}{{\mathscr B}}

\newcommand{\bbN}{{\mathbb N}}

\newcommand{\gK}{{\mathfrak K}}
\newcommand{\gA}{{\mathfrak A}}

\newcommand{\cJ}{{\mathscr J}}
\newcommand{\cL}{{\mathscr L}}
\newcommand{\bbC}{{\mathbb C}}

\newcommand{\cH}{{\mathscr H}}

\newcommand{\cI}{{\mathscr I}}

\newcommand{\bbL}{{\mathbb L}}

\newcommand{\bbP}{{\mathbb P}}
\newcommand{\cP}{{\mathscr P}}

\newcommand{\gy}{{\mathfrak y}}
\newcommand{\gx}{{\mathfrak x}}

\newcommand{\bbQ}{{\mathbb Q}}
\newcommand{\bbR}{{\mathbb R}}

\newcommand{\cS}{{\mathscr S}}
\newcommand{\cT}{{\mathscr T}}
 
\newcommand{\cU}{{\mathscr U}}

\newcommand{\cZ}{{\mathscr Z}}

\newcommand{\cf}{{\rm cf}}
\newcommand{\pr}{{\rm pr}}

\newcount\skewfactor
\def\mathunderaccent#1#2 {\let\theaccent#1\skewfactor#2
\mathpalette\putaccentunder}
\def\putaccentunder#1#2{\oalign{$#1#2$\crcr\hidewidth
\vbox to.2ex{\hbox{$#1\skew\skewfactor\theaccent{}$}\vss}\hidewidth}}
\def\name{\mathunderaccent\tilde-3 }

\newenvironment{PROOF}[2][\proofname.]
   {\begin{proof}[#1]\renewcommand{\qedsymbol}{\textsquare$_{\rm #2}$}}
   {\end{proof}}

\usepackage[hidelinks]{hyperref}

\begin{document}
\makeatletter\def\shfiuwefootnote{\gdef\@thefnmark{}\@footnotetext}\makeatother\shfiuwefootnote{Version 2017-04-07\_12. See \url{https://shelah.logic.at/papers/849/} for possible updates.}

\title {Beginning of stability theory for polish spaces \\
Sh849}
\author {Saharon Shelah}
\address{Einstein Institute of Mathematics\\
Edmond J. Safra Campus, Givat Ram\\
The Hebrew University of Jerusalem\\
Jerusalem, 91904, Israel\\
 and \\
 Department of Mathematics\\
 Hill Center - Busch Campus \\ 
 Rutgers, The State University of New Jersey \\
 110 Frelinghuysen Road \\
 Piscataway, NJ 08854-8019 USA}
\email{shelah@math.huji.ac.il}
\urladdr{http://shelah.logic.at}
\thanks{The author would like to thank the Israel Science Foundation for
partial support of this research (Grant No. 242/03). 
This was part of \cite{Sh:771} (first version 29/March/2001) which was
submitted 17/7/2002; but separated in Feb. 2009.  First version - 09/Feb/13}

% previous version - 2015/June/2

% split from sh771 - Feb 2009



\date{April 7, 2017}

\begin{abstract}  We consider stability theory for Polish spaces and
more generally for definable structures (say with elements of a set of
reals).  We clarify by proving some equivalent conditions for
$\aleph_0$-stability.  We succeed to prove
existence of indiscernibles under reasonable conditions; this gives
strong evidence that such a theory exists.
\end{abstract}

\maketitle
\numberwithin{equation}{section}
\setcounter{section}{-1}
\newpage

\section{Introduction} 

\subsection {General Aims} \
\bigskip

\begin{question}
\label{0.4}
Is there a stability theory/classification theory of Polish
spaces/algebras (more generally definable structures say on the continuum)?

Naturally we would like to develop a parallel to classification theory
and in particular stability theory (see \cite{Sh:c}).  
A natural test problem is to 
generalize ``Morley theorem = \L os conjecture". 
\underline{But} we only have one model so does it mean anything? 

On way is to restrict ourselves to a class of e.g. abelian groups and
replace ``categorical" by ``free"; this is reasonable as a free
algebra is unique in a strong sense (determined up to isomorphism by
its dimension which is equal to its cardinality when uncountable, well
when the vocabulary is countable).

We may consider any variety or just a theory consisting of 
universal Horn sentences (as then
being a free algebra is well defined);
this is an interesting question per se.  But still this is a special
case; and in general, we may change the universe.
\end{question}

\begin{example}
\label{0.5}
If $\bbP$ is adding $(2^{\aleph_0})^+$-Cohen subsets of $\omega$ then

\[
(\bbC)^{\mathbf V} \text{ and } (\bbC)^{\mathbf V[\bbP]}
\]

\mn
are both algebraically closed fields of characteristic $0$ which are
not isomorphic (as they have different cardinalities).

So we restrict ourselves to forcing notions $\bbP_1 \lessdot \bbP_2$ such that

\[
(2^{\aleph_0})^{\mathbf V[\bbP_1]} = (2^{\aleph_0})^{\mathbf V[\bbP_2]}
\]

\mn
and compare the Polish models in $\mathbf V^{{\bbP}_1},
\mathbf V^{{\bbP}_2}$.    We may restrict our forcing notions to
c.c.c. or whatever.
\end{example}

\begin{example}
\label{0.6}
Under any such interpretation
\mn
\begin{enumerate}
\item[$(a)$]   $\bbC =$ the field of complex numbers is categorical
\sn
\item[$(b)$]   $\bbR =$ the field of the reals is not 
(by adding $2^{\aleph_0}$ many Cohen reals).
\end{enumerate}
\mn
(Why?  Say we assume $\bbP_1 \lessdot \bbP_2,(2^{\aleph_0})^{\mathbf
V[\bbP_1]} = (2^{\aleph_0})^{\mathbf V[\bbP_2]}$ but $\bbR^{\mathbf
V[\bbP_1]} \ne \bbR^{\mathbf V[\bbP_2]}$. 
Trivially: $\bbR^{\mathbf V[{\bbP}_2]}$ is complete
in $\mathbf V[\bbP_2]$ while $\bbR^{\mathbf V[\bbP_1]}$ in 
$\mathbf V^{[{\bbP}_2]}$ is not complete, but there are less trivial reasons).
\end{example}

\begin{conjecture}
\label{0.7}
We have a dichotomy for ``reasonable definable model" (e.g. Polish or
Suslin ones), i.e. either the model
is similar to categorical theories, or there are ``many complicated
models" under the present interpretation.
\end{conjecture}

\noindent
So in particular we expect the natural variants of central notions
defined below (like categoricity) will be equivalent; in particular we
expect that it will be enough to consider the forcing notions of
adding Cohen reals.  Naturally those questions call for the use of
descriptive set theory on the one hand and model theory on the other
hand; in particular to using definability in both senses and using
$\bbL_{\aleph_1,\aleph_0}(\mathbf Q)$.

Presently, i.e. here there is no serious use of either; the questions
are naturally inspired by model theory.  It would be natural to
consider questions inspired by the investigation of such specific
structures; to some extent considering the freeness of a definable
Abelian group falls under this.

A priori, trying to connect different directions in mathematics is
tempting, but it may well lack non-trivial results.  We suggest that
the result on the existence of indiscernibility, \ref{7.18}, give
serious evidence that this is not the case here; note that even the weaker
\ref{7.4} gives more than set theory, i.e. Erd\"os-Rado theorem.  That
is, the question is: is there a non-trivial theory in this direction?
While the present work does not achieve a real theory of this kind, we
believe that it gives an existence proof.

Let us elaborate suggestions for the definition of ``categorical" and
for ``definable".  Recall that it is well known that: a Polish
structure is an $F_\sigma$-one which is a Borel one; a Borel (set or
structure) is a $\Sigma^1_1$-one, a $\Sigma^1_1$-set/structure is a
$\Sigma^1_2$-one and a
$\Sigma^1_2$-set/structure is a $\aleph_1$-Suslin one, see e.g. \cite{J}.
  
\begin{definition}
\label{z7f}
1) Let $\gA$ denote a definition of a $\tau$-structure, $\tau$ a
countable vocabulary, the set of elements of the structure is the reals or a
definable set of reals such that it is absolute enough, i.e. for any
forcing notion $\bbP \lessdot \bbQ$ we have $\gA[\mathbf V^{\bbP}]
\subseteq \gA[\mathbf V^{\bbQ}]$ but we may say ``the structure/model $\gA$".

\noindent
2) We say $\gA$ is $F_\sigma/\text{Borel}/\Sigma^1_1/\kappa$-Suslin,
etc., if the definition mentioned above is
$F_\sigma/\text{Borel}/\Sigma^1_1/\kappa$-Suslin, etc.

\noindent
3) We say $\gA$ is categorical$_1$ \when \, for any 
forcing notions $\bbP \lessdot \bbQ$ such that $\Vdash
   ``(2^{\aleph_0})^{\mathbf V[\bbP]} = (2^{\aleph_0})^{\mathbf V[\bbQ]}"$
the structures $\gA[\mathbf V^{\bbP}],\gA[\mathbf V^{\bbQ}]$ are isomorphic in
 $\mathbf V^{\bbQ}$.

\noindent
3A) We say $\gA$ is categorical$_1$ under $\varphi$ (e.g. $\varphi =
(2^{\aleph_0} = \aleph_1$) means that for any forcing notions
$\bbP \lessdot \bbQ$ satisfying 
$\Vdash_{\bbP} ``\varphi",\Vdash_{\bbQ} ``\varphi"$ and 
$\Vdash_{\bbQ} ``(2^{\aleph_0})^{\mathbf V[\bbP]} = 
(2^{\aleph_0})^{\mathbf V[\bbQ]}"$ the structures 
$\gA[\mathbf V^{\bbP}],\gA[\mathbf V^{\bbQ}]$ are isomorphic in
   $\mathbf V^{\bbQ}$.

\noindent
4) We say $\gA$ is $\gK$-categorical$_1$ \when \, above $\bbP,\bbQ \in
   \gK$ (or the pair $(\bbP,\bbQ) \in \gK$).

\noindent
5) If $\lambda$ is a cardinal or a definition of a cardinal \then \, $\gA$ is
categorical$_1$ in $\lambda$, is defined as in (3A) but $\Vdash_{\bbP}
   ``\varphi",\Vdash_{\bbQ} ``\varphi"$ is replaced by
$\lambda[\mathbf V^{\bbP}] = \lambda[\mathbf V^{\bbQ}] =
   (2^{\aleph_0})^{\mathbf V[\bbP]} = (2^{\aleph_0})^{\mathbf V[\bbQ]}$
   and this is a non-empty condition; similarly in (3A),(4).

\noindent
6) Let $T$ be a set of (first order) equations in the countable
   vocabulary $\tau$.  Let $\gA$ be a model of $T$.  We
   say $\gA$ is free$_1$ for $\gK$ \when \, for every $\bbP \in
\gK$ the model $\gA[\mathbf V^{\bbP}]$ is a free algebra.  Similarly, parallely
   to (3A),(4),(5).
\end{definition}

\begin{conjecture}
\label{z7k}
1) If, e.g. a $\Sigma^1_1$-structure $\gA$ 
is categorical$_1$ in some $\aleph_\alpha
\ge \aleph_\omega$ for $\gK$ \then \, it is categorical in every
 $\aleph_\alpha \ge \aleph_\omega$ of cofinality $> \aleph_0$ 
for $\gK$ where $\gK =
\{(\bbP,\bbQ):\bbP \lessdot \bbQ$ and $(2^{\aleph_0})[\mathbf V^{\bbP}] =
(2^{\aleph_0})[\mathbf V^{\bbQ}] \ge \aleph_{\omega}[\mathbf
V^{\bbP}]$, Card$[\mathbf V^{\bbP}] = \text{ Card}[\mathbf V^{\bbQ}]\}$.

\noindent
2) Or at least ``for every $\aleph_\alpha \ge \aleph_{\omega_1}"$
   and/or restricting the forcing notion to be c.c.c.

\noindent
3) Similarly for freeness.
\end{conjecture}

\begin{thesis}
\label{0.8}
1) Classification theory for such models resembles
more the case of $\mathbb L_{\omega_1,\omega}$ than the first order.

\noindent
2) If the continuum is too small we may get 
categoricity for ``incidental" reasons.

\noindent
3) In our context the parallel of categoricity and ``being in the low
   side of the main gap" seems to be equal.
\end{thesis}

\noindent
See \cite{Sh:h}; as support for this thesis, in \cite[\S5]{Sh:771} we prove:
\begin{theorem}
\label{0.9}
1) There is an $F_\sigma$-abelian group $G$ (i.e. an
$F_\sigma$-definition, in fact a very explicit definition) such that
$\mathbf V \models ``G$ is a free abelian group" iff $\mathbf V \models
2^{\aleph_0} < \aleph_{736}$.

\noindent
2) Moreover, it follows that if $\mathbf V \models ``G$ is not free" and
   $\bbP$ is a c.c.c. forcing notion \then \, the abelian group
   $G^{\mathbf V}$ is not free even in $\mathbf V^{\bbP}$.
\end{theorem}
\bigskip

\noindent
\underline{Comments}:  In the context of the previous theorem we cannot do
better than $F_\sigma$, but we may hope for some other examples which
is not a group or categoricity is not because of freeness.

\noindent
The proof gives
\begin{conclusion}
\label{0.10}
For any $n <\omega$ for some $F_\sigma$-abelian group $\gA,\gA$ is
categorical$_1$ in $\aleph_\alpha$ iff $\alpha \le n$.
\end{conclusion}

Note that if $\cf(\alpha) = \aleph_0$ so necessarily $\alpha > n$,
\then \, there are no forcing notions $\bbP,\bbQ$ as in \ref{z7f}(5),
hence categoricity fails by the definition.

A connection with the model theory is that by 
Hart-Shelah \cite{Sh:323} such things can also occur
in $\bbL_{\omega_1,\omega}$ whereas (by \cite{Sh:87a}, \cite{Sh:87b})
if $\bigwedge\limits_{n} (2^{\aleph_n} < 2^{\aleph_{n+1}})$ and
$\psi \in \mathbb L_{\omega_1,\omega}$ is categorical in every
$\aleph_n$, \then \, $\psi$ is categorical in every $\lambda$.  See
more in Shelah-Villaveces \cite{Sh:648} and in \cite{Sh:600}.  We 
may consider other interpretations of ``categoricity", more 
on \ref{0.8} - \ref{0.10}, see \S(0C) and \S4.

The parallels here are still open.

Those questions may cast some light on the thesis that non-first 
order logics are ``more distant" from the ``so-called" 
mainstream mathematics.  

This work originally was a section in \cite{Sh:771}; in it we try to look
at stability theory in this context, proving the modest (in
\ref{7.18}):
\mn
\begin{enumerate}
\item[$\boxplus$]  for ``$\aleph_0$-stable $\kappa$-Suslin models" 
the theorem on the existence of indiscernibles can be generalized.
\end{enumerate}
\mn
A natural question is
whether in the existence of indiscernibles (in Theorem \ref{7.18}) we
can start with a set of cardinality $\lambda$ rather than $\lambda^{+
\kappa^+}$.  Another natural question is whether categoricity implies
that the $\bbL_{\infty,\aleph_0}$-theory of $\gA$ is ``simple", clearly
by \S(0A) we are sure this holds.

Both questions (and maybe the one on the right parallel of stable)
may be addressed in a work in progress, \cite{Sh:F1134}.

We thank Udi Hrushovski and the referee for many helpful comments.
\bigskip

\subsection {The Content of the Present Work} \
\bigskip

Our context is a $\kappa$-candidate $(\gA,\Delta)$, so $\gA$ is
structure with the set of elements being a sets of reals, which is reasonably
definable, and to some extent similarly the relations the functions of
$\gA$ from $\Delta$ being reasonably definable: 
usually $\kappa$-Suslin; you can fix $\Delta$ as the set of
quantifier free formulas.  In the present work forcing does not
appear, so the universe is fixed.

In \S1; we give some basic definitions in Definition \ref{7.1} but note
that the $\aleph_0$-stable and $\aleph_0$-unstable are proved to
be complimentary only later.  The
main result is to get the ``end-extension indiscernibility existence"
lemma \ref{7.4}.  Using stability, we have that for any sequence
$\langle a_\alpha:\alpha <
\lambda\rangle$ we get a subsequence $\langle a_\alpha:\alpha \in
S\rangle,S \subseteq \lambda$ stationary such that the $\Delta$-type
of $\langle a_{\alpha_0},\dotsc,a_{\alpha_k}\rangle$ over
$\{a_\beta:\beta < \alpha_0\}$ does not depend on $\alpha_k$ when
$\alpha_0 < \ldots < \alpha_k$ are from $S$.  We then
improve it to ``does not depend" on
$\alpha_k,\alpha_{k-1},\dotsc,\alpha_{k-n}"$ for a fix $n$, but this
does not give full indiscernibility.  Note that the definition of
stability and its use in \ref{7.4} speak of definability of types but
the type are model theoretic ones, speaking only on $\Delta$-formulas
whereas the ``definable" means in set theoretic sense, so the
definition is arbitrary, just has to be in the submodel (of e.g. some
$(\cH(\chi),\in))$.  
This seems inherent in our framework: for a predicate $P \in
\tau(\gA)$, the set theoretic definition of $P^{\gA}$
may involve, as approximations, relations on the reals which are very
complicated model theoretically. 

In \S2 we generalize the theorem ``the order property implies unstability", in
\ref{7.7}; giving a criterion for unstability in \ref{7.7A}.  Now the
``unstable" in Definition \ref{7.1}(3) really speaks on having a
perfect set of types; we here define apparently weaker version
$(\mu,\Delta,\lambda)$-unstable.

Lastly, in \S3, we define ranks, but here the ranks are for subsets
$\mathbf I$ of ${}^m \gA$, not just definable ones; as explained above
this seems inherent in our framework.

We then prove (in \ref{7.14},\ref{7.17}) that ``$\gA$ is
$(\aleph_0,\Delta)$-stable" is really defined, i.e. that we have
several equivalent definitions: some from the structure side, some from
the non-structure side (generally on this see
\cite[\S(1A),\S(2A),\S(2B)]{Sh:E53}). 

Lastly, we prove a theorem on the existence of (fully) indiscernible
sub-sequence: a sequence of length $\lambda^{+ \kappa^+}$ has a
sub-sequence of length $\lambda$; a parallel situation occurs for strongly
dependent $T$ (by \cite{Sh:863}).  We end noting that, of course,
being $(\aleph_0,\Delta)$-unstable implies a failure of categoricity
(strong one) as expected in our frame.

On the existence of indiscernibles in stable models (not necessarily
of a stable theory) and history, see \cite{Sh:300b}.
\bigskip

\subsection {Further Comments} \
\bigskip

\begin{definition}
\label{0.12}
1) For a definition $\gA$ of a $\tau$-model (usually with a set of
   elements a definable set of reals) we say that $\gA$ is
categorical$_2$ in $\lambda < 2^{\aleph_0}$ \when \,: for some $A
\subseteq \lambda$: for every $A_1,A_2 \subseteq \lambda$ the 
models $\gA^{\mathbf L[A_1,A]},\gA^{\mathbf L[A_2,A]}$ 
are isomorphic (in $\mathbf V$). 

\noindent
2) For a class ${\gK}$ of forcing notions and cardinal $\lambda$
 we say $\gA$ is categorical$_2$ in $(\lambda,\gK)$ \when \, for
every $\bbP \in {\gK}$ satisfying\footnote{We may allow $\Vdash_{\bbP}
  ``2^{\aleph_0} \ge \lambda"$, but then for some $\bbP_1$ there is a
  $\bbP_2$-name $\name A \subseteq \lambda$ such that if $\bbP_1
\lessdot \bbP_2$ and $\bbP_2$-names $\name A_1,\name A_2$ of subset
 of $\lambda$ we have $\Vdash_{\bbP_2} 
``\name A^{\mathbf V[\name A_1,A_1]},\gA^{\mathbf V[\name A,
\name A_2]}$ are isomorphic".}
 $\Vdash_{\bbP} ``2^{\aleph_0} > \lambda"$, we have in 
$\mathbf V^{\bbP}$: the structure $\gA$ 
is categorical$_2$ in $\lambda$, i.e. in the sense of part (1).

\noindent
3) Note the $\lambda$ may be a given cardinal or definition of one.
\smallskip

Comparing Definition \ref{0.12}(1) with 
the forcing version we lose when $\mathbf V = \mathbf L$, as
it says nothing, we gain as (when $2^{\aleph_0} > \aleph_1$) we do not
have to go outside the universe.  
Maybe best is categorical in $\lambda$ in $\mathbf
V^{\bbP}$ for every c.c.c. forcing notion $\bbP$ making
$2^{\aleph_0} > \lambda$, see \S4.  Of course, the examples from
\ref{0.9} works also here.

Note also that it may be advisable in \ref{0.12}(1) 
to restrict ourselves to the case
$\lambda$ is regular as we certainly like to avoid the possibility
$(2^{\aleph_0})^{\mathbf L[A_1,r]} = \lambda < (2^{\aleph_0})^{\mathbf
L[A_2,r]}$ (see on this and for history in \cite[Ch.VII]{Sh:g}).
\end{definition}

Of course, any reasonably absolute definition of unstability implies
non-categoricity: if we have many types we should have a perfect set
of them, hence adding Cohen subsets of $\omega$ adds more types
realized.  If we add $\langle \eta_i:i < 2^{\aleph_0} \rangle$ Cohen
reals then $\langle \gA^{\mathbf V[\eta_i:i \in
A]}:A \subseteq 2^{\aleph_0} \rangle$ is a sequence of 
pairwise non-isomorphic over the
countable set of parameters, if we get $2^{2^{\aleph_0}}$ or just $>
2^{\aleph_0}$ non-isomorphic models, we can forget the parameters and retain our
``richness in models".  The present work is to some extent a continuation of
\cite{Sh:202}, \cite{Sh:522}, see history there.

\begin{remark}
\label{0.15}
For $\ell(*)=1,2$ letting $\kappa = \aleph_{\ell(*)-1}$, we 
may replace everywhere $\kappa$-Suslin, $\kappa,\kappa^+$ by
$\Sigma^1_{\ell(*)},\aleph_{\ell(*)-1},\aleph_{\ell(*)}$  respectively.
Presently, though the first (i.e. 
$\Sigma^1_1$) case is more restrictive than the
second $(\aleph_1$-Suslin) case, what we can prove is the same.  For
$\ell(*)=1$ we have real equivalence.
\end{remark}

\noindent
Note that below the indiscernible sequences are indiscernible sets,
justified by \ref{7.7} (see, e.g. \cite{Sh:c}).
\begin{claim}
\label{z20}
1) Assume $I$ is an infinite linear order $\bar a_s \in {}^m \gA$ for
   $s \in I,\langle \bar a_s:r \in I\rangle$ is a
$(\Delta,n)$-indiscernible sequence (see \ref{7.4a}) but not a
$(\Delta,n)$-indiscernible set (see \ref{7.4a}(9)).  \Then \, 
we can find $\varphi = 
\varphi(\bar x_0,\dotsc,\bar x_n,\bar y) \in \Delta$ and $\bar b
   \in {}^{\ell g(\bar y)} A$ and $\ell < n$ such that:  if $s_0 <_I
   \ldots <_I s_n$ then
\mn
\begin{enumerate}
\item[$(a)$]  $M \models \varphi[\bar a_{s_0},\dotsc,\bar
a_{s_n},\bar b]$
\sn
\item[$(b)$]  $M \models \varphi[\bar a_{t_0},\dotsc,\bar a_{t_n},\bar b]$
when
\sn
\begin{enumerate}
\item[$\bullet$]  $t_k = s_k$ if $k \ne \ell,\ell +1$
\sn
\item[$\bullet$]  $t_k = s_{k+1}$ if $k =\ell$
\sn
\item[$\bullet$]  $t_k = s_{k-1}$ if $k = \ell +1$.
\end{enumerate}
\end{enumerate}
\mn
2) So if $I = \lambda$, this implies the $(\lambda,\Delta)$-order
property.

\noindent
3) Similarly for end-$(\Delta,n)$-indiscernible sequence (but we may
   have to demand $s_0 \ge s_* \in I$ for a fix $s_*$.
\end{claim}

\noindent
Recall
\begin{definition}
\label{z17}
1) Let $\mathbf T_\kappa$ be the following tree:
\mn
\begin{enumerate}
\item[$(a)$]  members have the form $t= (\nu_t,\eta_t)$ such that for
some $n,\eta_t \in {}^n \omega,\nu_t \in {}^n \kappa\}$
\sn
\item[$(b)$]  it is ordered by $s \le_{\mathbf T_\kappa} t$ iff $\eta_s
\trianglelefteq \eta_t \wedge \nu_s \trianglelefteq \nu_t$.
\end{enumerate}
\mn
2) For a subtree $\cT$ of $\mathbf T_\kappa$ let $\proj(\cT) = \{\eta \in
{}^\omega \omega$: for some $\nu \in {}^\omega \kappa$ we have
$(\nu,\eta) \in \lim(\cT)\}$.

\noindent
3) A set of this form is called $\kappa$-Suslin.

\noindent
4) For notational convenience we may replace ${}^\omega \omega$ by
   ${}^\omega(\cH(\aleph_0))$ and consider $\bar a \in {}^m(\gA)$ as
$\langle \langle a_\ell(n):\ell < \ell g(\bar a)\rangle:n
   <\omega\rangle$ and let $\bar a \upharpoonleft n = \large< \langle
a_\ell(k):\ell < \ell g(\bar a)\rangle:k < n\large>$; 
(similarly we may replace ${}^\omega \kappa$ by
   ${}^\omega(\mathbf L_\kappa)$, too).
\end{definition}
\newpage

\section {Generalizing stability in $\aleph_0$}

We may consider the dividing line for abelian groups 
from \cite{Sh:402} and try to
generalize it for any simply defined (e.g. anywhere from a Polish
algebra until $\kappa$-Suslin) model.  We
deal with trying to prove that we have two possibilities, in the 
high, complicated
side we get a parallel of non $\aleph_0$-stability (hence strong
non-categoricity, see \ref{c19}); in the low side we
have a rank.  But even for minimal formulas, the example in
\cite[\S5]{Sh:771} shows
that we are far from being done, still we may be able to say something on the
structure.

We may consider also ranks parallel to the ones for superstable
theories.  Recall that there are two kinds of definability we are
considering: the model theoretic one and the set theoretic one.  

\begin{context}
\label{7.0}
1) If not said otherwise, ${\gA}$ will
be a structure with countable vocabulary and its set of elements is a
set of reals; usually a definition - see \ref{z7f}(1).

\noindent
2) ${\cL}$ is a logic.  We did not specify the logic; we may assume it is 
$\subseteq \bbL_{\aleph_1,\aleph_0}$ or just $\bbL_{\aleph_1,\aleph_0}
(\mathbf Q)$ where $\mathbf Q$ is the quantifier ``there
are uncountably many".

\noindent
3)  $\kappa$ is a fix cardinality, let a $\kappa$-model mean a
$\kappa$-Suslin model.
\end{context}

\begin{definition}
\label{a6}
1)  For a structure ${\gA}$, an ${\gA}$-formula or
$(\gA,\cL)$-formula $\varphi$ is a formula in the language 
${\cL}(\tau_{\gA})$ so in the logic $\cL$, and the formula
is in the vocabulary of 
${\gA}$ with finitely many free variables, writing $\varphi =
\varphi(\bar x)$ means that $\bar x$ is a finite sequence of variables with no
repetitions including the free variables of $\varphi$. 

\noindent
2) $\Delta$ denotes a set of such formulas and $\bar \varphi$ denotes a pair
$(\varphi_0(\bar x),\varphi_1(\bar x))$ of formulas so $\bar \varphi$
is a $\Delta$-pair when $\varphi_0,\varphi_1 \in \Delta$; in general
we do not require any closure properties on $\Delta$. 

\noindent
3) We say $\varphi$ (or $\Delta$ or $\bar \varphi$) is $\kappa$-Suslin
(or $\Sigma^1_1$ or $\Sigma^1_2$ or $\Delta^1_0$ (= Borel)) \Iff \, they 
are so as set theoretic formulas, see \ref{a9}(2).

\noindent
4) We say $\gA$ is a $\kappa$-model or a $\kappa$-Suslin model \when
 \, being a member, (that is, in its universe) 
its relation and its functions are
   $\kappa$-Suslin; so the definition of $\gA$ is the sequence of
   $\kappa$-Suslin definitions and saying, e.g. $\gA \in N \subseteq
(\cH(\chi),\theta)$ means that also this sequence of
   $\kappa$-Suslin definitions belongs to $N$.
\end{definition}

\begin{discussion}
\label{a9}
1) So if the relation ``$\eta \in \gA$", ``$(\eta_0,\dotsc,\eta_{n-1})
\in P^{\gA}$" where $P$ is an $n$-place predicate\footnote{pedantically we
   should change a little in \ref{z17}, or we can translate $P^{\gA}
   \subseteq {}^n({}^\omega \omega)$ to a subset of ${}^\omega
   \omega$} of $\tau_{\gA}$, are $\kappa$-Suslin, \then \, they are
   upward and downward absolute (as long as the relevant subtrees of
   $\mathbf T_\kappa$ are in the universe).

\noindent
2) When we consider a formula $\varphi(\bar x)$ this is more
complicated but we are assuming it, too, is $\kappa$-Suslin, meaning that
   this (set theoretic) definition defines $\varphi(\bar x)^{\gA}$
   also in the other universes we consider.
\end{discussion}

\begin{definition}
\label{7.1}
1) We say $({\gA},\Delta)$ is a $\kappa$-candidate (or $\kappa$-Suslin
candidate; as $\kappa$ is constant we may omit it) \when \,:
\mn
\begin{enumerate}
\item[$(a)$]   ${\gA}$ is a $\kappa$-model
\sn
\item[$(b)$]   $\Delta$ is a countable set of
$({\gA},{\cL})$-formulas\footnote{We may use ``$\Delta$ is a set of
pairs $\bar\varphi = (\varphi_0(\bar x),\varphi_1(\bar x)$ of such
formulas", and later demand $\bar\varphi \in \Delta$.  So far it
does not matter.}
which, are in the set theory sense, $\kappa$-Suslin (we 
identify $\varphi$ and $\neg \neg \varphi$), i.e. $\Delta$ contains
the information saying which $\kappa$-Suslin relations are used
\sn
\item[$(c)$]  we consider changes of the universe (say by forcing)
only when this is preserved, anyhow we can assume each $\varphi(\bar x)
\in \Delta$ is quantifier free.
\end{enumerate}
\mn
1A) We can replace being $\kappa$-Suslin by being $\Sigma^1_1$, 
by $\Sigma^1_2$, etc., (naturally we need enough absoluteness); if we
replace it by $\Gamma$ we write $\Gamma$-candidate.  If $\Gamma$ does
not appear we mean it is $\kappa$-Suslin or understood from the context.

\noindent
2) If $({\gA},\Delta)$ is a candidate we say ${\gA}$ is 
$(\aleph_0,\Delta)$-stable (or $({\gA},\Delta)$ is
$\aleph_0$-stable), \when \, $\Delta$ is
a countable set of $({\gA},{\cL})$-formulas and for 
$\chi$ large enough and for every countable
$N \prec ({\cH}(\chi),\in,<^*_\chi)$ to which $(\gA,\Delta)$ 
belongs and $\bar a \in {}^m {\gA}$ where $m < \omega$,
the following weak definability condition on tp$_\Delta(\bar a,N
\cap {\gA},{\gA})$ holds:
\mn
\begin{enumerate}
\item[$\boxplus$]  letting $\Phi^m_{\gA,\Delta}$ be defined in $(*)$,
for some function $\mathbf c \in N$ with domain
$\Phi^m_{({\gA},\Delta)}$ and range $\subseteq \{0,1\}$ 
the statement $(**)$ holds where
\sn
\begin{enumerate}
\item[${{}}$]   $(*) \quad$ letting $\Phi^m_{{\gA},\Delta} =
\Phi^m_{({\gA},\Delta)} = 
\{\bar \varphi(\bar x,\bar b):\bar \varphi(\bar x,\bar b) =
(\varphi_0(\bar x,\bar b), \varphi_1(\bar x,\bar b))$ 

\hskip25pt and $\bar x = \langle x_\ell:\ell < m
\rangle,\bar b \in {}^{\omega >}{\gA}$ and
$\varphi_0,\varphi_1 \in \Delta$ and

\hskip25pt  ${\gA} \models \neg(\exists \bar x,\bar y)
(\varphi_0(\bar x,\bar y) \wedge \varphi_1(\bar x,\bar y))\}$, 
\sn
\item[${{}}$]  $(**) \quad$ if $\bar \varphi = (\varphi_0(\bar x,\bar
b),\varphi_1(\bar x,\bar b)) \in \Phi^m_{({\gA},\Delta)} \cap N$
and $\ell < 2$ and 

\hskip25pt ${\gA} \models \varphi_\ell[\bar a,\bar b]$
\then \, $\ell = \mathbf c(\bar \varphi)$.
\end{enumerate}
\end{enumerate}
\mn
2A) We say $(\gA,\Delta)$ is $\mu$-stable \when \, above we replace
``$N$ countable" by ``$N$ of cardinality $\mu$ and $\mu +1 \subseteq
N"$.

\noindent
3) We say that $({\gA},\Delta)$ is $\aleph_0$-unstable (or ${\gA}$ 
is $(\aleph_0,\Delta)$-unstable) \when \,:
there are $\bar a_\eta \in {}^m {\gA}$ for $\eta \in {}^\omega 2$ 
and $\varphi_{\nu,0}(\bar x,\bar y_\nu) \in \Delta$ and 
$\varphi_{\nu,1}(\bar x,\bar y_\nu) \in 
\Delta$ and $\bar b_\nu \in {}^{\ell g(\bar y)} {\gA}$ for
$\nu \in {}^{\omega >} 2$ such that:
\mn
\begin{enumerate}
\item[$(a)$]   ${\gA} \models \neg(\exists \bar x)
(\varphi_{\nu,0}(\bar x,\bar b_\nu) \wedge \varphi_{\nu,1}(\bar x,\bar b_\nu))$
\sn
\item[$(b)$]    if $\nu \triangleleft \eta_0,
\nu \triangleleft \eta_1$ and $\eta_0,\eta_1 \in {}^\omega 2,
n = \ell g(\nu)$ and $\eta_0(n) = 0,\eta_1(n)=1$ \then \, 
${\gA} \models \varphi_{\nu,0}[\bar a_{\eta_0},\bar b_\nu] 
\wedge \varphi_{\nu,1}[\bar a_{\eta_1},\bar b_\nu]$.
\end{enumerate}
\mn
4) Let $\bar\varphi(\bar x,\bar y) \in \Phi^m_{(\gA,\Delta)}$ or
``$\bar\varphi(\bar x,\bar y)$ is a $\Delta$-pair" (pedantically
$(\gA,\Delta)$-pair) means
that $\bar\varphi(\bar x,\bar y) = (\varphi_0(\bar x,\bar
y),\varphi_1(\bar x,\bar y))$ and $\varphi_0(\bar x,\bar
y),\varphi_1(\bar x,\bar y)$ belongs to $\Delta$ and $\gA \models
\neg(\exists \bar x,\bar y)(\varphi_0(\bar x,\bar y) \wedge
\varphi_1(\bar x,\bar y))$.

\noindent
5) Let $\bar\varphi(\bar x,\bar y) \in \Delta$ means that
   $\bar\varphi(\bar x,\bar y) = (\varphi_0(\bar x,\bar
   y),\varphi_1(\bar x,\bar y))$ and $\varphi_0(\bar x,\bar
   y),\varphi_1(\bar x,\bar y)$ belongs to $\Delta$.
\end{definition}

\begin{remark}
\label{a15}
1) There are obvious absoluteness results (for $\bar \varphi \in
\Phi^m_{({\gA},\Delta)},({\gA},\Delta)$ is $\aleph_0$-unstable and
$\aleph_0$-stable).

\noindent
2) On those notions (stable and unstable) being complimentary 
see Theorem \ref{7.14}.

\noindent
3) We may consider in Definition \ref{7.1}(1) using countable models
or models of cardinality $\kappa$, it is not clear to us what is
better but see \ref{7.2} below.
\end{remark}

\begin{observation}
\label{7.2}
1) If $\Delta$ is closed under negation \then \, in Definition 
\ref{7.1}(2) we equivalently can replace $\boxplus$ by:
\mn
\begin{enumerate}
\item[$\boxplus'$]   for some $\mathbf c \in N$ we have: if $\varphi(\bar x,\bar
y) \in \Delta$ and $\bar b \in {}^{\ell g(\bar y)}{\gA}$ and $\bar b
\in N$ \then \, ${\gA} \models \varphi(\bar a,\bar b)$ iff
$\mathbf c(\varphi(\bar x,\bar b))=1$.
\end{enumerate}
\mn
2) In Definition \ref{7.1}(2) for all $\chi > 2^{\aleph_0}$ the
following variants of the definition are equivalent:
\mn
\begin{enumerate}
\item[$(a)$]  the original one
\sn
\item[$(b)$]  as there but omitting ``countable" and
$<^*_\chi$, (a well ordering of $\cH(\chi)$)
\sn
\item[$(c)$]  there is a club of countable $N \subseteq
  (\cH(\chi),\in)$ which satisfies the conclusion (i.e. there $\mathbf c
  \in N$ such that ...)
\sn
\item[$(d)$]  for each $\Delta$-pair $\bar\varphi(\bar x,\bar y)$
  there is a club $\cS \subseteq [\cH(\chi)]^{\aleph_0}$ such that for
every $u \in \cS$ letting $N_u = (\cH(\chi),\in) \rest A$ we have:
  for some function $\mathbf c \in N_A$
\sn
\begin{enumerate}
\item[$(\alpha)$]  $\Dom(\mathbf c) = \{\bar\varphi(\bar x,\bar b):\bar
  b \in {}^{\ell g(\bar y)}\gA\},\Range(\mathbf c) \subseteq \{0,1\}$
\sn
\item[$(\beta)$]  for every $\bar b \in N_A \cap {}^m \gA$ and $\ell <
  2$ if $\gA \models \varphi_\ell[\bar a,\bar b]$ then $\ell = \mathbf
  c(\bar\varphi(\bar x,\bar b)$.
\end{enumerate}
\sn
\item[$(e)$]  like clause (d) replacing ``club $\cS$" by ``stationary
  $\cS$".
\end{enumerate}
\mn
3) In Definition \ref{7.1}(2), the demand ``$N$ is countable" can be
omitted (and in \ref{7.1}(2A) can weaken $\|N\| = \mu$ to $\|N\| \ge \mu$).
\end{observation}

\begin{PROOF}{\ref{7.2}}
Straight; e.g. $(a) \Rightarrow (b)$ let 
$N^* \prec ({\cH}(\chi),\in,<^*_\chi)$ be uncountable and such
that ${\gA},\Delta \in N^*$.  Now for every countable $N \prec N^*$ to
which $({\gA},\Delta)$ belongs there is $\mathbf c_N \in N$ as 
mentioned in the definition \ref{7.1}.
Hence by normality of the club filter on
$[N^*]^{\aleph_0}$, the family of countable subsets of $N^*$,
for some $\mathbf c^*$ the set $\cS' =
\{N:N \prec N^*$ is countable and $\mathbf c_N = \mathbf c^*\}$ is a stationary
subset of $[N^*]^{\aleph_0}$, so $\mathbf c^*$ can serve for $N^*$. 
\end{PROOF}

\begin{claim}
\label{7.4}
\underline{The End-Extention Indiscernibility existence lemma}   
1) Assume:
\mn
\begin{enumerate}
\item[$(a)$]  $(\alpha) \quad \Delta$ is closed under negation
  \underline{or} just
\sn
\item[${{}}$]  $(\beta) \quad$ if $\gA \models ``\neg \varphi(\bar
  a,\bar b)"$ and $\varphi(\bar x,\bar y) \in \Delta$ then for some
  $\bar\varphi(\bar x,\bar y) \in \Phi^{\ell g(\bar x)}_{\gA,\Delta}$

\hskip25pt  we have $\varphi_0 = 
\varphi$ and $\gA \models \varphi_1[\bar a,\bar b]$
\sn
\item[$(b)$]  $\Delta$ is closed under permuting the (free)
variables, $m < \omega$ 
\sn
\item[$(c)$]   $(\gA,\Delta)$ is an $\aleph_0$-candidate which is
  $\aleph_0$-stable (or just $\mu$-stable, $\mu < \lambda$, see \ref{7.1}(2A)
\sn
\item[$(d)$]  $\aleph_0 < \lambda = \text{\rm cf}(\lambda)$ and $S
\subseteq \lambda$ is stationary
\sn
\item[$(e)$]  $\bar a_\alpha \in {}^m {\gA}$ for $\alpha < \lambda$
\sn
\item[$(f)$]  $A \subseteq {\gA}$ has cardinality $< \lambda$.
\end{enumerate}
\mn
\Then \, for some stationary $S' \subseteq
S$ the sequence $\langle \bar a_\alpha:\alpha \in S' \rangle$ is a
$\Delta$-end extension indiscernible sequence over $A$ in ${\gA}$
(see Definition \ref{7.4a}(3),(4),(5) below). 

\noindent
2) Moreover for any pregiven $n < \omega$ we can find stationary $S'
\subseteq S$ such that $\langle \bar a_\alpha:\alpha \in S' \rangle$
is $(\Delta,n)$-end extension indiscernible over $A$ in ${\gA}$. 

\noindent
3) In part (2) we can find a club $E$ of 
$\lambda$ and regressive function $f_n$
on $S \cap E$ for $n < \omega$ such that:
\mn
\begin{enumerate}
\item[$(i)$]    if $\alpha,\beta \in S \cap E$ \then \,
$f_{n+1}(\alpha) = f_{n+1}(\beta) \Rightarrow f_n(\alpha) =
f_n(\beta)$
\sn
\item[$(ii)$]    if $n < \omega$ and $\gamma < \lambda$, \then \,
the sequence $\langle \bar a_\alpha:\alpha
\in S \cap E,f_n(\alpha) = \gamma \rangle$ is $(\Delta,n)$-end
extension indiscernible sequence over $A$
\sn
\item[$(ii)^+$]   moreover, if $n < \omega$ and $\beta,\gamma < \lambda$
\then \, $\langle \bar a_\alpha:\alpha  \in S \cap E \backslash \beta$ and
$f_n(\alpha) = \gamma \rangle$ is $(\Delta,n)$-end extension
indiscernible sequence over $\cup \{\bar a_\gamma:\gamma < \beta\}
\cup A$.
\end{enumerate}
\end{claim}

\begin{remark}
\label{7.4y}
1) This is a ``first round on indiscernibility".

\noindent
2) The assumption ``$\Delta$ is closed under negation" is quite
strong.

\noindent
3) Really we get indiscernible sets by \ref{z20} and \ref{7.7}.

\noindent
4) The claim and proof are similar to 
\cite[Ch.III,4.23,pg.120-1]{Sh:c}.  But before proving we define:
\end{remark}

\begin{definition}
\label{7.4a}
1) Let $({\gA},\Delta)$ be a candidate.  
We say ``${\gA}$ has $(\lambda,\Delta)$-order" \when \,:
\mn
\begin{enumerate}
\item[$(*)_\lambda$]   for some $m(*) < \omega$ and 
$\bar \varphi(\bar x,\bar y) \in \Phi^{m(*)}_{{\gA},\Delta}$ with 
$\ell g(\bar x) = \ell g(\bar y)$, this formula linear orders some 
${\mathbf I} \subseteq {}^{m(*)} {\gA}$ of cardinality $\lambda$, see
part (2) for definition. 
\end{enumerate}
\mn
2) We say that the formula $\bar \varphi(\bar x,\bar y)$ linear orders $\mathbf I
\subseteq {}^{m(*)}{\gA}$ 
\Iff \, for some $\langle \bar a_t:t \in I \rangle$ we have:
\mn
\begin{enumerate}
\item[$(a)$]   $\mathbf I = \{\bar a_t:t \in I\}$
\sn
\item[$(b)$]   $I$ is a linear order
\sn
\item[$(c)$]   $\bar \varphi = (\varphi_0(\bar x,\bar y),
\varphi_1(\bar x,\bar y)) \in \Phi^m_{(\gA,\Delta)}$
\sn
\item[$(d)$]   if $s <_I t$ then ${\gA} 
\models \varphi_0[\bar a_s,\bar a_t] \wedge \varphi_1[\bar a_t,\bar a_s]$.
\end{enumerate}
\mn
3) For a linear order $J$ (e.g. a set of ordinals), we say $\langle
\bar a_t:t \in J \rangle$ is a $\Delta$-end-extension indiscernible
sequence (over $A$) \Iff \, for 
any $n < \omega$ and $t_0 <_J < \ldots <_J t_{n-1} <_J t_n$, 
the sequences $\bar a_{t_0} \char 94 \ldots 
\bar a_{t_{n-2}} \char 94 \bar a_{t_{n-1}}$ and $\bar a_{t_0} \char 94
\ldots \char 94 \bar a_{t_{n-2}} \char 94 \bar a_{t_n}$ realize the
same\footnote{We may consider ``not containing a contradictory pair".}
 $\Delta$-type (over $A$) in ${\gA}$. 

\noindent
4) We say that $\langle \bar a_t:t \in J \rangle$ is
$(\Delta,n_0,n_1)$-end-extension indiscernible sequence over $A$ in ${\gA}$
\when \,:
\mn
\begin{enumerate}
\item[$(a)$]   $J$ a linear order for some $m,\bar a_t \in
{}^m{\gA},A \subseteq {\gA}$
\sn
\item[$(b)$]   if $\langle r_\ell:\ell < n_0 \rangle,\langle
s_\ell:\ell < n_1 \rangle,\langle t_\ell:\ell < n_1 \rangle$ are
$<_J$-increasing sequences and $n_0 > 0 \wedge n_1 > 0$ implies 
$r_{n_0-1} <_J s_0,r_{n_0-1} <_J t_0$ \then 
\newline
$\bar a_{r_0} \char 94 \ldots \char 94 \bar a_{r_{n_0-1}}
\char 94 \bar a_{s_0} \char 94 \ldots \char 94 \bar a_{s_{n_1-1}}$ and
$\bar a_{r_0} \char 94 \ldots \char 94 \bar a_{r_{n_0-1}} \char 94
\bar a_{t_0} \char 94 \ldots \char 94 \bar a_{t_{n_1-1}}$ realizes the
same $\Delta$-type over $A$ in ${\gA}$
\sn
\item[$(c)$]   if $J$ has a last element we allow to decrease $n_0$
and/or $n_1$.
\end{enumerate}
5) If we omit $n_0$ this means for every $n_0$, (so ``$\Delta$-end
extension..." means $(\Delta,1)$-end extension.

\noindent
6) We say $\langle \bar a_t:t \in I\rangle$ is
   $(\Delta,n)$-indiscernible sequence over $A$ when it is an
   $(\Delta,0,n)$-indiscernible sequence over $A$.

\noindent
7) We say $\langle \bar a_t:t \in I\rangle$ is a
   $\Delta$-indiscernible sequence over $A$ \when \, it is an
$(\Delta,n)$-indiscernible sequence over $A$ for every $n$.

\noindent
8) In (4) if we replace above ``sequence" by ``set"; this means that for
   every permutation $\pi$ of $n_1$ in clause (4), also $\bar a_{r_0}
   \char 94 \ldots \char 94 \bar a_{r_{n_0-1}} \char 94 \bar
   a_{s_{\pi(0)}},\dotsc,\bar a_{s_{\pi(n_1-1)}}$ realizes the same
$\Delta$-type over $A$ in $\gA$; similarly in clause (c).

\noindent
9) Like (8) for (5), (6), (7).
\end{definition}

\begin{PROOF}{\ref{7.4}}
\underline{Proof of \ref{7.4}}  

\noindent
1) Let $\langle N_\alpha:\alpha < \lambda \rangle$ be an increasing
continuous sequence of elementary submodels of $({\cH}(\chi),
\in,<^*_\chi)$ to which ${\gA}$ belongs, such that
$\|N_\alpha\| < \lambda,N_\alpha \cap \lambda \in \lambda$ and $\alpha
\subseteq N_\alpha$ and $\langle \bar a_\alpha:\alpha < \lambda
\rangle \in N_0$ (hence $\bar a_\alpha \in N_{\alpha +1}$).  
For each $\alpha \in S$, applying \ref{7.2}(3)
to $N_\alpha,\bar a_\alpha$ we get $\mathbf c_\alpha \in N_\alpha$ as in
Definition \ref{7.1}(2).  
So for some $\mathbf c^*$ and some stationary subsets of $S' \subseteq S$ of
$\lambda$ we have $\alpha \in S' \Rightarrow \mathbf c_\alpha =
\mathbf c^*$. 

Now let $\beta_1 < \beta_2$ be from the set $S'$ and we shall prove
that $p_1 \subseteq p_2$ when we define $p_\ell = 
\tp_\Delta(\bar a_{\beta_\ell},\cup\{\bar a_\alpha:
\alpha < \beta_\ell\} \cup A,\gA)$ for
$\ell=1,2$.  So assume $\varphi = \varphi(\bar x,\bar y) \in
\Delta,\ell g(\bar x) = m(*)$ and $\bar b$ is a sequence of length
$\ell g(\bar y)$ from $\cup\{\bar a_\alpha:\alpha < \beta_1\} \cup A$
and $\ell \in \{1,2\}$ and assume $\gA \models \neg \varphi[\bar
  a_{\beta_\ell},\bar b]$ and it suffices to prove the $\gA \models
\neg \varphi[\bar a_{\beta_{3-\ell}},\bar b)$.

Now there is $\bar\varphi \in \Phi^{m(*)}_{\gA,\Delta}$ such that
$\varphi_0 \equiv \varphi$ and $\gA \models \varphi_1[\bar
  a_{\beta_\ell},\bar b]$.  Why?  Apply clause (a) of the assumption
of our claim \ref{7.4}.  If $(\alpha)$ there holds, that is, $\Delta$
is closed under negation choose $\bar\varphi(\bar x,\bar y) =
(\varphi_0(\bar x,\bar y),\neg \varphi_0(\bar x,\bar y))$, it is a
$(\gA,\Delta)$-pair and necessarily $\gA \models \varphi_1[\bar
  a_{\beta_\ell},\bar b]$; otherwise, that is sub-clause $(\beta)$
there holds so there is a $(\gA,\Delta)$-pair $\bar\varphi$ such that
$\varphi_0 = \varphi$ and $\gA \models \varphi_1[\bar
  a_{\beta_\ell},\bar b]$.

In both cases necessarily $\mathbf c_{\beta_\ell}(\bar\varphi(\bar
a_{\beta_\ell},\bar b)) = 1$ but $\mathbf c_{\beta_1} = \mathbf
c_{\beta_2}$ hence $\mathbf c_{\beta_{3-\ell}}(\bar\varphi(\bar
a_{\beta_{3-\ell}},\bar b))=1$ which recalling
$\boxplus(**)$ in Definition \ref{7.1} implies $\mathbf
c_{\beta_{3-\ell}}(\bar\varphi(\bar a_{\beta_{3-\ell}},\bar b) \ne 0$ 
so by the choice of
$\bar\varphi$ we have $\gA \models ``\neg \varphi_0[\bar
  a_{\beta_{3-\ell}},\bar b]"$ which means $\gA \models ``\neg
\varphi[\bar a_{\beta_{3-\ell}},\bar b]"$ as promised, so we are done.

\noindent
2) We prove this by induction on $n$, that is, we prove by induction
on $n$ that:
\mn
\begin{enumerate}
\item[$\boxtimes^n_\lambda$]   for all $m < \omega$,
a stationary $S \subseteq \lambda$ and $\bar a_\alpha \in {}^m{\gA}$ for
$\alpha < \lambda$ 
\underline{there is} a stationary $S' \subseteq S$ such that: 
if $\beta < \lambda,\alpha'_\ell \in S',\alpha''_\ell \in S'$ for
$\ell < n$ and $\beta \le \alpha'_0 < \alpha'_1 < \ldots$ and $\beta
\le \alpha''_0 < \alpha''_1 < \ldots$ then $\bar a_{\alpha'_0} \char
94 \ldots \bar a_{\alpha'_{n-1}},\bar a_{\alpha''_0} \char 94 \ldots
\char 94 \bar a_{\alpha''_{n-1}}$ realizes the same $\Delta$-type over 
$\cup\{\bar a_\gamma:\gamma < \beta\} \cup A$.
\end{enumerate}
\mn
For $n=0$ the demand is empty so $S'=S$ is as required.  For $n=1$
apply the proof of part (1).  
For $n+1 > 1$, by the induction hypothesis we can find a
stationary $S_1 \subseteq S$ as required in $\boxtimes^n_\lambda$.  For
each $\alpha < \lambda$ we can choose $\beta_{\alpha,\ell} =
\beta(\alpha,\ell)$ for $\ell \le n$ such that $\alpha =
\beta_{\alpha,0} < \beta_{\alpha,1} < \ldots < \beta_{\alpha,n}$ and
$0 < \ell \le n \Rightarrow \beta_{\alpha,\ell} \in S_1$.  Let $\bar
a^*_\alpha = \bar a_{\beta_{\alpha,0}} \char 94 \ldots \char 94 \bar
a_{\beta_{\alpha,n}}$ so $\bar a^*_\alpha \in {}^{m(n+1)}{\gA}$
and apply the induction hypothesis or just the case $n=1$, i.e. part (1)
 to $m \times (n+1),S_1,\langle \bar
a^*_\alpha:\alpha < \lambda \rangle$ getting a stationary $S_2
\subseteq S_1$ as required in $\boxtimes^n_\lambda$ or just in
$\boxtimes^1_\lambda$. 

\Wilog \,
\mn
\begin{enumerate}
\item[$(*)_1$]  for every $\gamma \in S_2$ and $\alpha \in \gamma$ 
we have $\beta_{\alpha,n} < \gamma$ hence
\sn
\item[$(*)_2$]  if $\gamma \in S_2$ then $\bigcup\limits_{\alpha <
  \gamma} \bar a^*_\alpha \subseteq \bigcup\limits_{\alpha < \gamma}
\bar a_\alpha$, moreover, $\alpha < \gamma \wedge 0 < \ell \le n
\Rightarrow \beta_{\alpha,\ell} < \gamma$.
\end{enumerate}
\mn
We claim that $S_2$ is as required.  So assume $\beta \le \alpha'_0 <
\ldots < \alpha'_n < \lambda$ and $\beta \le \alpha''_0 < \ldots
< \alpha''_n < \lambda$ and $\alpha'_\ell,\alpha''_\ell \in S_2$ for
$\ell \le n$.  

Now, recalling $\beta(\alpha,\ell) = \beta_{\alpha,\ell}$ for $\ell \le
n,\alpha < \lambda$ we have:
\mn
\begin{enumerate}
\item[$(i)$]    $\bar a_{\alpha'_0} \char 94 \bar a_{\alpha'_1} \char 94
\ldots \char 94 \bar a_{\alpha'_n}$ and $\bar a_{\alpha'_0} \char 94
\bar a_{\beta(\alpha'_0,1)} \char 94 \ldots \char 94 \bar
a_{\beta(\alpha'_0,n)}$ realizes the same $\Delta$-type over $\cup
\{\bar a_\gamma:\gamma < \beta\} \cup A$ in ${\gA}$.
\end{enumerate}
\mn
[Why?  As $\ell \in \{1,\dotsc,n\} \Rightarrow
\beta(\alpha'_1,\ell) \in S_1 \wedge (\alpha'_\ell \in S_2 \subseteq
S_1)$ and apply the choice of $S_1$ with $\alpha'_0 +1$ here standing
for $\beta$ there.]
\mn
\begin{enumerate}
\item[$(ii)$]   $\bar a_{\alpha'_0} \char 94 \bar
a_{\beta(\alpha'_0,1)} \char 94 \ldots \char 94 \bar
a_{\beta(\alpha'_0,n)}$ is equal to $\bar a^*_{\alpha'_0}$.
\end{enumerate}
\mn
[Why?  By the choice of $\bar a^*_{\alpha'_0}$.]
\mn
\begin{enumerate}
\item[$(iii)$]    $\bar a^*_{\alpha'_0},\bar a^*_{\alpha''_0}$
realizes the same $\Delta$-type over $\cup\{\bar a^*_\gamma:\gamma
< \beta\} \cup A$ hence over $A \cup \{\bar a_\gamma:\gamma < \beta\}$.
\end{enumerate}
\mn
[Why?  By the choice of $S_2$]. 

Similarly
\mn
\begin{enumerate}
\item[$(iv)$]    $\bar a^*_{\alpha''_0}$ is equal to
$\bar a_{\alpha''_0} \char 94 \bar
a_{\beta(\alpha''_0,1)} \char 94 \ldots \char 94 \bar
a_{\beta(\alpha''_0,n)}$ 
\sn
\item[$(v)$]    $\bar a_{\alpha''_0}
\char 94 \bar a_{\beta(\alpha''_0,1)} \char 94 \ldots \char 94 \bar
a_{\beta(\alpha''_0,n)}$ and $\bar a_{\alpha''_0} \char 94 \bar a_{\alpha''_1}
\char 94 \ldots \char 94 \bar a_{\alpha''_n}$ 
realizes the same $\Delta$-type over $\cup\{\bar a_\gamma:\gamma <
\beta\} \cup A$.
\end{enumerate}
\mn
By (i)-(v) recalling 
$(*)_2$ the set $S_2$ is as required in $\boxtimes^{n+1}_\lambda$. 

\noindent
3) The proofs of parts (1), (2) actually give this.
\end{PROOF}

\begin{claim}
\label{b32}
In Theorem \ref{7.4}(3) we can add:
\mn
\begin{enumerate}
\item[$(iii)$]  we have (where $e_n$ and $\langle S_{n,\gamma}:\gamma
  < \lambda\rangle$ are defined below)
\sn
\begin{enumerate}
\item[$(\alpha)$]  if $\gamma_1,\gamma_2 < \lambda$ are
  $e_n$-equivalent and $S_{n+1,\gamma_1},S_{n+1,\gamma_0}$ are
stationary \then \, for some $\Delta$-formula $\varphi(\bar
  x_0,\dotsc,\bar x_n,\bar y)$ and sequence $\bar b$ of length $\ell
  g(\bar y)$ from $\cup\{\bar a_\alpha:\alpha <
  \min\{\gamma_1,\gamma_2\}\} \cup A$ we have:
\sn
\item[${{}}$]  $\bullet \quad$ if $\alpha_0 < \ldots < \alpha_n$ are
  from $S_{n+1,\gamma_1} \cup S_{n+1,\gamma_2}$ then $\gA \models
  \varphi[\bar a_{\alpha_0},\dotsc,\bar a_{\alpha_n},\bar b]$ 

\hskip25pt  iff $\alpha_0 \in S_{n+1,\gamma_2}$ where
\sn
\item[$(\beta)$]   $e_n$ is the equivalent relation
  $\{(\gamma_1,\gamma_2)$: there are $\alpha_1 > \gamma_1$ and
  $\alpha_2 > \gamma_2$ (but $< \lambda$) such that $f_n(\alpha_1) =
  f_n(\alpha_2)$ but $f_{n+1}(\alpha_1) = \gamma_1,f_{n+1}(\alpha_2) =
  \gamma_2\}$
\sn
\item[$(\gamma)$]   $S_{n,\gamma} = \{\alpha < \lambda:\alpha >
  \gamma$ and $f_n(\alpha) = \gamma\}$.
\end{enumerate}
\end{enumerate}
\end{claim}

\begin{PROOF}{\ref{b32}}
By induction on $n$.  We choose $f'_n$, a club $E'_n$ and so
$e_n,\langle S_{n,\gamma}:\gamma < \lambda\rangle$ such that:
\mn
\begin{enumerate}
\item[$\boxplus$]  $(a) \quad E_n \subseteq E$
\sn
\item[${{}}$]  $(b) \quad E'_n,\langle f'_\ell:\ell \le n\rangle$ are
  as required in \ref{7.4}(3)
\sn
\item[${{}}$]  $(c) \quad S_{n,\gamma}$ are defined as in
  $(iii)(\gamma)$ above from $f'_n$
\sn
\item[${{}}$]  $(d) \quad e_\ell$ for $\ell < n$ is
 defined from $f'_\ell,f'_{\ell +1}$ as in $(\beta)$ above.
\end{enumerate}
\mn
There are no problems.
\end{PROOF}
\newpage

\section {Order and unstability}

\begin{lemma}
\label{7.7}
\underline{The order/unstability lemma}:

Assume that
\mn
\begin{enumerate}
\item[$\boxtimes_1$]   $(a) \quad
({\gA},\Delta)$ is a $\kappa$-candidate, (e.g. $\ell(*) \in
\{1,2\},\kappa = \aleph_{\ell(*)-1}$ the relation

\hskip25pt  are $\Sigma^1_{\ell(*)}$ hence a $\kappa$-Suslin)
\sn
\item[${{}}$]  $(b) \quad \varphi_0(\bar x,\bar y),\varphi_1(\bar x,\bar y)
\in \Delta$ are contradictory in ${\gA}$
\sn
\item[${{}}$]  $(c) \quad J$ is a linear order of cardinality $\lambda$
\sn
\item[${{}}$]   $(d) \quad \bar a_t \in {}^m {\gA}$ for $t \in J$
\sn
\item[${{}}$]  $(e) \quad {\gA} \models 
\varphi_0[\bar a_s,\bar a_t]$ and $\gA \models \varphi_1[\bar a_t,\bar a_s]$

\hskip25pt whenever $s <_J t$
\sn
\item[$\boxtimes_2$]   $(a) \quad \lambda \ge \aleph_{\kappa^+}$ 
\underline{or} 
\sn
\item[${{}}$]  $(b) \quad J$ is with density $\mu < |J|$ 
and $\kappa < |J|$.
\end{enumerate}
\mn
\Then \, $({\gA},\Delta)$ is $\aleph_0$-unstable; 
even more specifically the demand in Definition \ref{7.1}(3)
holds with $\varphi_{\nu,0} = \varphi_0,\varphi_{\nu,1} = \varphi_1$.
\end{lemma}

\begin{question}
\label{b6}
What can $\{\lambda:{\gA}$ has a $(\Delta,\lambda)$-order, that is
$\boxplus_1$ of \ref{7.7} holds for $J=\lambda\}$ be?
\end{question}

\begin{remark}
\label{b9}
1) Note that we cannot hope for too much (for \ref{b6}), because a
   specific case is:

$\{\lambda$: in the Boolean algebra $\cP(\bbN)/$ finite there is an
   increasing chain of length $\lambda\}$.

\noindent
2) However, if we replace $J \cong \lambda$ by $|J|=\lambda$, then we
   have something to say.
\end{remark}

\noindent
We first prove a claim from which we can derive the lemma.
\begin{claim}
\label{7.7A}
If $\boxplus$ below holds, then the pair
$({\gA},\Delta)$ is $\aleph_0$-unstable; moreover upward absolutely so
(and downward if $(\gA,\Delta)$ are in the submodel) \where \,:
\mn
\begin{enumerate}
\item[$\boxplus$]  $(a) \quad ({\gA},\Delta)$ is a
$\kappa$-candidate and $m < \omega,\Phi = 
\Phi^m_{({\gA},\Delta)}$ or just $\Phi \subseteq
\Phi^m_{({\gA},\Delta)}$
\sn
\item[${{}}$]   $(b) \quad \bar{\cP} = \langle {\cP}_\alpha:\alpha <
\kappa^+ \rangle$
\sn
\item[${{}}$]   $(c) \quad {\cP}_\alpha$ is a 
non-empty family of subsets of ${}^m{\gA}$
\sn
\item[${{}}$]   $(d) \quad$ if $\alpha < \beta < \kappa^+$ and $\mathbf I \in 
{\cP}_\beta$ \then \, for some $\mathbf I_0,\mathbf I_1 \in {\cP}_\alpha$
which are

\hskip25pt  subsets of $\mathbf I$ and some pair
$(\varphi_0(\bar x,\bar b),\varphi_1(\bar x,\bar b)) \in \Phi$ we have

\hskip25pt $\ell < 2$ and $\bar a \in \mathbf I_\ell \Rightarrow {\gA} \models
\varphi_\ell(\bar a,\bar b)$
\sn
\item[${{}}$]   $(e) \quad$ if $\mathbf I \in {\cP}_\beta$ and $\alpha < \beta <
\kappa^+$ and $F$ is a function with domain $\mathbf I$ and

\hskip25pt  range of cardinality $\le \kappa$, \then \, there is $\mathbf I' \in
{\cP}_\alpha$ such that

\hskip25pt  $\mathbf I' \subseteq \mathbf I$
and $F \restriction \mathbf I'$ is constant.
\end{enumerate}
\end{claim}

\begin{PROOF}{\ref{7.7}}
\underline{Proof of \ref{7.7} from \ref{7.7A}}
\bigskip

\noindent
\underline{Case 1}:  $\lambda \ge \aleph_{\kappa^+}$ 

For $\alpha < \kappa^+$ we let

\[
{\cP}_\alpha = \{\mathbf I:\mathbf I \subseteq {}^m {\gA} \text{ is
linearly ordered by } \bar \varphi \text{ and has cardinality }
\ge \kappa^{+ \alpha}\}.
\]
\bigskip

\noindent
\underline{Case 2}:  $J$ has density $\mu,|J| > \mu + \kappa$

For $\alpha < \kappa^+$ we let

\begin{equation*}
\begin{array}{clcr}
{\cP}_\alpha = \{\mathbf I:&\mathbf I \subseteq {}^m {\gA} \text{ is
linearly ordered by } \bar \varphi \text{ getting an order of
cardinality} \\
  &> \mu + \kappa \text{ and density } \le \mu\}.
\end{array}
\end{equation*}

\mn
This should be clear.

By $\boxtimes_2$ of the assumption, at least one of the cases holds.
\end{PROOF}

\begin{PROOF}{\ref{7.7A}}
\underline{Proof of \ref{7.7A}}  

For each $\varphi(\bar x) \in \Delta$
as $\{\bar a \in {}^{\ell g(\bar x)} {\gA}:{\gA} \models
\varphi[\bar a]\}$ is a $\kappa$-Suslin set and\footnote{Pedantically
note that we sometimes consider $\bar a \in {}^{\ell g(\bar x)}\gA$ as
a member of ${}^\omega \omega$ or better ${}^\omega(\cH(\aleph_0))$
 rather than ${}^{\ell g(\bar x)}({}^\omega \omega)$ or see \ref{z17}(4).} 
\mn
\begin{enumerate}
\item[$(*)_{\varphi(\bar x)}$]  we can find 
$\langle \cT_\varphi:\varphi \in \Delta\rangle$ such that (recalling
  \ref{z17}(4)) 
\sn
\begin{enumerate}
\item[$(a)$]  if $\alpha < \kappa$ then $\cT_\varphi$ is a closed subtree of 
$\mathbf T_\kappa$, see \ref{z17}   
\sn
\item[$(b)$]  $\{\bar a \in {}^{\ell g(\bar x)}{\gA}:{\gA} \models
\varphi[\bar a]\} = \{\bar a$: some $\nu \in {}^\omega \kappa$ witness
$\varphi[\bar a]$, that is $(\bar a,\nu) \in \lim(\cT_\varphi)\}$.
\end{enumerate}
\end{enumerate}
\mn
We can find functions $F_\varphi$ 
such that if $\varphi(\bar x) \in \Delta$ and ${\gA} \models
\varphi[\bar a]$ then $F_\varphi(\bar a) \in {}^\omega \kappa$
witnesses this.
For notational simplicity and \wilog \, $m=1$. Let $W =
\{w:w \subseteq {}^{\omega >} 2$ is a front\footnote{i.e. for every
$\eta \in {}^\omega 2$ there is one and only one $n < \omega$ such
that $\eta \rest n \in w$} hence finite$\}$.

For $w \in W$ and $n < \omega$ let 
$Q_{n,w}$ be the family of objects ${\gx} = (w,n,\bar u,\bar
\nu,\bar\varrho,\bar\varphi) = (w_{\gx},n_{\gx},\ldots)$ such that:
\mn
\begin{enumerate}
\item[$(*)_{\gx}$]   for unboundedly many $\alpha < \kappa^+$ 
we can find a witness (or an $\alpha$-witness)
$\mathbf y = (\langle \bar a_\ell:\ell < n \rangle,\langle B_\rho:\rho \in w
\rangle)$ for $\gx$ which means:
\sn
\begin{enumerate}
\item[$(a)$]  $(\alpha) \quad \bar u = \langle (u^0_\rho,u^1_\rho):
\rho \in w \rangle$ 
\sn
\item[${{}}$]  $(\beta) \quad$ if $\rho \in w$ then
$u^0_\rho,u^1_\rho \subseteq n$ are disjoint
\sn
\item[${{}}$]  $(\gamma) \quad \bar \varphi = \langle
 \bar\varphi^\ell:\ell < n \rangle$ where $\bar \varphi^\ell =
(\varphi^\ell_0(\bar x,\bar y_\ell),\varphi^\ell_1(\bar x,\bar y_\ell)) \in 
\Phi^m_{(\gA,\Delta)}$
\sn
\item[$(b)$]  $\bar a_\ell \in {}^{\ell g(\bar y_\ell)}{\gA}$
\sn
\item[$(c)$]   $B_\rho \in {\cP}_\alpha$ for $\rho \in w$
\sn
\item[$(d)$]   if $\rho \in w,\bar b \in B_\rho$ and $\ell < n$ then
$(\varphi^\ell_0(\bar x,\bar a_\ell),\varphi^\ell_1(\bar x,\bar a_\ell)) \in
\Phi^m_{(\gA,\Delta)},\ell g(\bar x) = m,\ell g(\bar y_\ell)$ 
arbitrary (but finite) and
\sn
\item[${{}}$]  $(\alpha) \quad \ell \in u^0_\rho 
\Rightarrow {\gA} \models \varphi^\ell_0[\bar b,\bar a_\ell]$
\sn
\item[${{}}$]  $(\beta) \quad \ell \in u^1_\rho 
\Rightarrow {\gA} \models \varphi^\ell_1[\bar b,\bar a_\ell]$
\sn
\item[$(e)$]   if $\nu \ne \rho$ are from $w$ then 
$(u^0_\rho \cap u^1_\nu \ne \emptyset) \vee (u^0_\nu \cap
u^1_\rho \ne \emptyset)$
\sn
\item[$(f)$]    $(\alpha) \quad \bar \nu = \langle \nu^i_{\rho,\ell}:\rho \in
w,i \in \{0,1\}$ and $\ell \in u^i_\rho \rangle$
\sn
\item[${{}}$]  $(\beta) \quad \bar\varrho = \langle
\varrho^i_{\rho,\ell}:\rho \in w,i \in \{0,1\}$ and $\ell \in
u^i_\rho\rangle$
\sn
\item[${{}}$]  $(\gamma) \quad \nu^i_{\rho,\ell} \in {}^{\omega
  >}\omega$ pedantically $\in {}^{\omega >}\cH(\aleph_0)$
\sn
\item[${{}}$]  $(\delta) \quad \varrho^i_{\rho,\ell} \in {}^{\omega >}\kappa$
\sn
\item[$(g)$]   if $\rho \in w,\bar b \in B_\rho,i \in 
\{0,1\},\ell \in u^i_\rho$ then $\varrho^i_{\rho,\ell} \triangleleft
F^0_{\varphi^\ell_i}(\bar b,\bar a_\ell)$ and $\nu^i_{\rho,\ell}
\triangleleft (\bar b \char 94 \bar a_\ell)$, see \ref{z17}(4).
\end{enumerate}
\end{enumerate}
\mn
Clearly
\mn
\begin{enumerate}
\item[$\boxplus_1$]    $Q_{0,\{<>\}} \ne \emptyset$. 
\end{enumerate}
\mn
[Why?  Let ${\gx}$ be such that $w_{\gx} = \{\langle \rangle\},n_{\gx} = 0,\bar
  u_{\gx} = \langle \rangle,\bar\varrho_{\gx} = \langle
  \rangle,\bar\nu_{\mathbf x} = \langle \rangle$ and
  $\bar\varphi_{\gx} = \langle \rangle$
and if $\alpha < \kappa^+$ choose $\mathbf I \in {\cP}_\alpha$ we let 
$B_{<>} = \mathbf I$; recall that by clause (c) of Claim \ref{7.7A} the
family $\cP_\alpha$ is non-empty.]

Next (the aim is to increase the length of the $\nu$'s (so
$\gy = \gx$ cannot work)
\mn
\begin{enumerate}
\item[$\boxplus_2$]   if ${\gx} \in Q_{n,w}$ and $\rho \in w$ \then \,
there is $\gy$ such that:
\sn
\begin{enumerate}
\item[$\bullet$]  $\gy \in Q_{n,w}$
\sn
\item[$\bullet$]  $\bar u_{\gy} = \bar u_{\gx}$
\sn
\item[$\bullet$]  $\nu^i_{\gx,\rho,\ell} \triangleleft
\nu^i_{\gy,\rho,\ell}$
\sn
\item[$\bullet$]  $\varrho^i_{\gy,\rho,\ell} \triangleleft
\varrho^i_{\gx,\rho,\ell}$
\sn
\item[$\bullet$]  $\bar\varphi_{\gy} = \bar\varphi_{\gx}$.
\end{enumerate}
\end{enumerate}
\mn
[Why?  As ${\gx} \in Q_{n,w}$ we know that for some unbounded
$Y \subseteq \kappa^+$ for 
each $\alpha \in Y$ there is an $\alpha$-witness $\mathbf y_\alpha = (\langle
\bar a^\alpha_\ell:\ell < n \rangle,\langle B^\alpha_\rho:\rho \in 
w \rangle)$ as required in $(*)_{\gx}$.  
Let $\alpha < \kappa^+$ and $\beta(\alpha) = \text{ Min}
(Y \backslash (\alpha +1))$.  Now for each $\rho \in w$ we have 
$B^{\beta(\alpha)}_\rho \in {\cP}_{\beta(\alpha)}$.

Let $k > \sup\{\ell g(\nu^i_{\gx,\rho,\ell}),\ell
g(\varrho^i_{\gx,\rho,\ell}):i < 2,\rho \in w$ and
$\ell < n\}+1$ and, of course, the set
$\{F_{\varphi^\iota_{\gx,\ell}}(\bar b,\bar a^{\beta(\alpha)}_\ell) \rest
k:\bar b \in B^{\beta(\alpha)}_\rho\}$ has cardinality $\le \kappa$ 
for each $i < 2,\iota < 2,\rho \in w,\ell \in u^i_{\gx,\rho}$.  
Hence by clause (e) of Claim \ref{7.7A} recalling $\alpha < \beta(\alpha)$
we can find a subset $B_{\alpha,\rho}$ of
$B^{\beta(\alpha)}_\rho$ from $\cP_\alpha$ and
$\nu^{\alpha,i}_{\rho,\ell} \in {}^{\omega >}\cH(\aleph_0),
\varrho^{\alpha,i}_{\rho,\ell} \in {}^{\omega >}\kappa$
such that
\mn
\begin{enumerate}
\item[$\bullet$]   $\bar b \in B_{\alpha,\rho} \Rightarrow 
(\bar b \char 94 \bar a_{\gx,\ell}) \upharpoonleft k = \nu^\iota_{\rho,\ell}$
\sn
\item[$\bullet$]  $\bar b \in B_{\alpha,\rho} \Rightarrow
F_{\varphi^\ell_{\gx,\iota}}(\bar b,\bar a_{\gx,\ell}) \rest k
= \varrho^{\alpha,i}_{\rho,\ell}$.
\end{enumerate}
\mn
Similarly for some $\langle \nu^{*,i}_{\rho,\ell},
\varrho^{*,i}_{\rho,\ell}:i < 2,\rho \in w,\ell \in 
u^i_{\gx,\rho}\rangle$ we have
\mn
\begin{enumerate}
\item[$\bullet$]  $Y' = \{\alpha \in Y:\nu^{\alpha,i}_{\rho,\ell} =
\nu^{*,\iota}_{\rho,\ell}$ and $\varrho^{\alpha,i}_{\rho,\ell} =
\rho^{*,i}_{\rho,\ell}$ for $\rho \in w,\iota \in \{0,1\}$ end $\ell
\in u^i_\rho\}$ is unbounded in $\kappa^+$.
\end{enumerate}
\mn
Now it is easy to choose $\gy$.]
\mn
\begin{enumerate}
\item[$\boxplus_3$]    if ${\gx} \in Q_{n,w}$ and $\rho \in w$ and
$v = (w \backslash \{\rho\}) \cup \{\rho \char 94 \langle 0 \rangle,
\rho \char 94 \langle 1 \rangle\}$ so $v \in W$, 
\then \, there is ${\mathbf y} \in Q_{n+1,v}$ such that:
\sn
\begin{enumerate}
\item[$(\alpha)$]   $\bullet \quad n_{\mathbf y} = n_{\gx} +1$
\sn
\item[$(\beta)$]  $\bullet \quad u^i_{\mathbf y,\eta} =
u^i_{\gx,\eta}$ for $\eta \in w \backslash \{\rho\},i = 0,1$
\sn
\item[${{}}$]  $\bullet \quad u^i_{\mathbf y,\rho \char 94 <j>} \cap
\{0,\dotsc,n-1\} = u^i_{\gx,\rho}$ for $i = 0,1$ and $j=0,1$
\sn
\item[${{}}$]  $\bullet \quad n \in u^i_{\mathbf y,\rho \char 94 <j>}
  \Leftrightarrow j = i$ for $i,j \in \{0,1\}$
\sn
\item[$(\gamma)$]  $\bullet \quad \nu^i_{\mathbf y,\eta,\ell} =
\nu^i_{\gx,\eta,\ell}$ for $\eta \in w \backslash \{\rho\}$
\sn
\item[${{}}$]  $\bullet \quad \nu^i_{\mathbf y,\rho \char 94 <\iota>,\ell} =
\nu^i_{\gx,\rho,\ell}$
\sn
\item[$(\delta)$]  $\bullet \quad \varrho^i_{\mathbf y,\eta,\ell} =
\varrho^i_{\gx,\eta,\ell}$ for $\eta \in w \backslash \{\rho\}$ and
$\ell < m$
\sn
\item[${{}}$]  $\bullet \quad \varrho^i_{\mathbf y,\rho \char 94 <\iota>,\ell}
= \varrho^i_{\gx,\rho,\ell}$ for $\ell < n$
\sn
\item[${{}}$]  $\bullet \quad \varrho^i_{\mathbf y,\rho \char 94 <\iota>,\ell}
= \langle \rangle$ for $i = \iota,\ell=n_{\gx}$
\sn
\item[$(\varepsilon)$]  $\bar\varphi_{\gy} \rest n =
\bar\varphi_{\gx}$.
\end{enumerate}
\end{enumerate}
\mn
[Why?  Now we use clause (d) of the assumption of Claim \ref{7.7A} but
  we shall elaborate.  Choose $Y$ and $\mathbf y_\alpha = (\langle
  a^\alpha_\ell:\ell < n\rangle,\langle B^\alpha_\rho:\rho \in w\rangle)$ for
  $\alpha \in Y$ as in the proof of $\boxplus_2$.  For each $\alpha
< \kappa^+$ let $\beta(\alpha) = \min(Y \backslash (\alpha +1)))$ and
we apply clause (d) of $\oplus$ of \ref{7.7A} with
  $\alpha,\beta(\alpha),B^{\beta(\alpha)}_\rho$ here standing for
$\alpha,\beta,\mathbf I$ there.

So there are $\bar\varphi(\bar x,\bar c),\mathbf I_0,\mathbf I_1$ such
that:
\mn
\begin{enumerate}
\item[$\boxplus_{3.1}$]  $(a) \quad \bar\varphi(\bar x,\bar c) \in
  \Phi^m_{(\gA,\Delta)}$
\sn
\item[${{}}$]  $(b) \quad \mathbf I_0,\mathbf I_1 \subseteq
  B^{\beta(\alpha)}_\rho$
\sn
\item[${{}}$]  $(c) \quad \mathbf I_0,\mathbf I_1 \in \cP_\alpha$
\sn
\item[${{}}$]  $(d) \quad$ if $j < 2$ and $\bar a \in \mathbf I_j$ then
  $\gA \models \varphi_j[\bar a,\bar c]$.
\end{enumerate}
\mn
Now $\gy$ is almost defined in $\boxplus_3$, the missing points are:
\mn
\begin{enumerate}
\item[$\bullet$]  we choose $\bar\varphi_{\gy,n}(\bar x,\bar y_n)$ as
  $\bar\varphi$
\sn
\item[$\bullet$]  we choose $\bar a_{\gy,n}$ as $\bar c$.
\end{enumerate}
\mn
Now it suffices for every $\alpha < \kappa$ to choose an
$\alpha$-witness $\mathbf y'_\alpha$ of $\gy$, it is defined by
\mn
\begin{enumerate}
\item[$\bullet$]  $\bar a_\ell = a^{\beta*(\alpha)}_\ell$ for $\ell < n$
\sn
\item[$\bullet$]  $\bar a_\ell = \bar c$ for $\ell < n$
\sn
\item[$\bullet$]  $B_\eta$ is $B^{\beta(\alpha)}_\eta$ if $\eta
  \in \omega \backslash \{\rho\}$ and is $\mathbf I_j$ if $\eta = \rho
  \char 94 \langle j \rangle,j < 2$.
\end{enumerate}
\mn
Now check that $\gy$ is as required in $\boxplus_3$.]
\mn
\begin{enumerate}
\item[$\boxplus_4$]  we can choose $w_n,\rho_n$ by induction on $n$ such that
\sn
\begin{enumerate}
\item[$(a)$]  $w_n$ is a front of ${}^{n>}2$, included in ${}^k 2 \cup
  {}^{k+1}2$ for some $k = k(n)$
\sn
\item[$(b)$]  $w_0 = \{\langle \rangle\}$
\sn
\item[$(c)$]  $w_{n+1} = (w_n \backslash \{\rho_n\}) \cup \{\rho_n
  \char 94 \langle 0 \rangle,\rho_n \char 94 \langle 1 \rangle\}$ for
  some $\rho_n$
\sn
\item[$(d)$]  if $w_n \cap {}^{k(n)}2 \ne \emptyset$ then $\rho_n \in
  w_n \cap {}^{k(n)}2$ and $k(n+1) = k(n)$
\sn
\item[$(e)$]  if $w_n \cap {}^{k(n)}2 = \emptyset$ then $\rho_n \in
  w_n = {}^{k(n)+1}2$ and $k(n+1) = k(n)+1$
\sn
\item[$(f)$]  necessarily $\{\rho_n:n < \omega\} = {}^{\omega >}2$.
\end{enumerate}
\end{enumerate}
\mn
[Why?  Trivially.]
\mn
\begin{enumerate}
\item[$\boxplus_5$]  we can choose $\gx^1_n,\gx^2_n$ by induction on
  $n$ such that
\sn
\begin{enumerate}
\item[$(a)$]  $\gx^1_n,\gx^2_n \in Q_{n,w_n}$
\sn
\item[$(b)$]  $\gx^2_n$ is gotten from $\gx^1_n$ by $\boxplus_2$
\sn
\item[$(c)$]  $\gx^1_{n+1}$ is gotten from $\gx^2_n$ by $\boxplus_3$
  for $\rho_n$.
\end{enumerate}
\end{enumerate}
\mn
[Why?  For $n=0$, choose $\gx^1_n$ by $\boxplus_1$ and then $\gx^2_n$
 by $\boxplus_2$.  For $n=k+1$ choose $\gx^1_n$ satisfying (c) by
 $\boxplus_3$ for $\rho_n$ and then choose $\gx^2_n$ satisfying (b) by
 $\boxplus_2$.]
\mn
\begin{enumerate}
\item[$\boxplus_6$]  we can choose continuous functions $\rho \mapsto
  \bar a_\rho$ and $\rho \mapsto \varrho_{\rho,n}$ for $n <\omega$
  where $\rho$ varies on ${}^\omega 2$ such that:
\sn
\begin{enumerate}
\item[$(a)$]  $\bar a_\rho \in {}^m \gA$
\sn
\item[$(b)$]  $\varrho_{\rho,\ell} \in {}^\omega \kappa$
\sn
\item[$(c)$]  if $\rho_1 \ne \rho_2 \in {}^\omega 2$ and we let $\rho
  = \rho_1 \cap \rho_2,k = \ell g(\rho)$ and $n$ is such that $\rho_n
  = \rho$ \then \, letting $\iota(\ell) = \rho_\ell(k)$ so $\iota(1)
  \ne \iota(2)$ are $< 2$, \then \, for $\ell=1,2$ we have:
\sn
\item[${{}}$]  $\bullet \quad \varrho_{\rho_\ell,n}$ witness $\gA
  \models \varphi^{\iota(\ell)}_{\gx_{n+1},n}(\bar a_\rho,\bar
  c_{\gx_{n+1},n})$.
\end{enumerate}
\end{enumerate}
\mn
[Why?  See the $(*)_{\gx_{n+1},\ell}$ and the choice of the $\gx$'s.]

Clearly $\boxplus_6$ gives the desired conclusion (even absolutely).
\end{PROOF}

\begin{remark}
\label{7.8}  
This claim can be generalized replacing $\aleph_0$
by $\mu$, strong limit singular of cofinality $\aleph_0$. 
\end{remark}

\centerline {$* \qquad * \qquad *$}

\begin{definition}
\label{7.9}
Let $(\gA,\Delta)$ be a $\kappa$-candidate and $m < \omega$.

\noindent
1) For $A \subseteq \gA$ and $\bar a \in {}^n \gA$ let
 $\tp_\Delta(\bar a,A,{\gA}) = \{\varphi(\bar x,\bar b):
\varphi(\bar x,\bar y) \in \Delta$ and $\bar b \in {}^{\ell g(\bar y)}(A)$ and 
${\gA} \models \varphi[\bar a,\bar b]\}$.  

\noindent
2) $\Phi_{{\gA},\Delta,A}^{\text{pr},m} = 
\{(\varphi_0(\bar x,\bar b),\varphi_1(\bar x,\bar b)):\varphi_0(\bar
x,\bar y),\varphi_1(\bar x,\bar y)$ belong to $\Delta$ and $\bar b
\in {}^{\ell g(\bar y)} A$ and $\bar x = \langle x_\ell:\ell < m
\rangle$ and ${\gA} \models \neg(\exists \bar x)[\varphi_0(\bar
x,\bar b) \wedge \varphi_1(\bar x,\bar b)]\}$ 
where $A \subseteq {\gA},\Delta$ a set of ${\gA}$-formulas,
and so $\Phi^m_{\gA,\Delta} = \Phi^{\pr,m}_{\gA,\Delta,A}$ when $A$ is the
universe of $\gA$.

\noindent
3) ${\mathbf S}^m_\Delta(A,{\gA}) = \{\text{tp}_\Delta(\bar a,A,{\gA}):
\bar a \in {}^m {\gA}\}$ where $A \subseteq {\gA}$ and
$\Delta$ a set of $\bbL(\tau_{\gA})$-formulas.
\end{definition}

\begin{definition}
\label{7.10}
1) We say $({\gA},\Delta)$ is
$(\mu,\Delta,\lambda)$-unstable \when \, there are $M \subseteq
{\mathfrak A},m < \omega$ and $\langle \bar a_\alpha:\alpha < \lambda \rangle$
such that:
\mn
\begin{enumerate}
\item[$(a)$]   $\bar a_\alpha \in {}^m{\gA}$
\sn
\item[$(b)$]   if $\alpha \ne \beta$ are $< \lambda$ then for some
$(\varphi_0(\bar x,\bar b),\varphi_1(\bar x,\bar b)) \in
\Phi^m_{\gA,\Delta,M}$ (see Definition \ref{7.9}(2)) 
we have $\varphi_0(\bar x,\bar b) \in 
\text{ tp}_\Delta(\bar a_\alpha,M,{\gA})$ 
and $\varphi_1(\bar x,\bar b) \in \text{ tp}_\Delta(\bar a_\beta,M,{\gA})$
\sn
\item[$(c)$]    $\|M\| \le \mu$.
\end{enumerate}
\mn
1A) Let ${\gA}$ be $(\aleph_0,\Delta,\text{per})$-unstable mean
that $({\gA},\Delta)$ is $\aleph_0$-unstable; here per stands for perfect.

\noindent
2) In part (1) and (1A) we add ``weakly" if we weaken clause (b) to
\mn
\begin{enumerate}
\item[$(b)^-$]   tp$_\Delta(\bar a_\eta,M,{\gA}) 
\ne \text{ tp}_\Delta(\bar a_\nu,M,{\gA})$ for $\eta \ne \nu$ from $X$ 

(so if $\Delta$ is closed under negation there is no difference); in
part (1), $X = \lambda$ and in part (1A), $X = {}^\omega 2$.
\end{enumerate}
\mn
3) We use $(\mu_0,\Delta,x,\bbQ)$ where $\bbQ$ is a forcing notion
\Iff \, the example is found in
$\mathbf V^{\bbQ}$ such that usually $M$ is in $\mathbf V$ and we add 
an additional possibility if $x = \text{ per}$ then $M \in \mathbf V$
and $X = ({}^\omega 2)^{\mathbf V}$ (here per stands for perfect).

\noindent
4) We may replace ``a forcing notion $\bbQ$" by a family ${\gK}$
of forcing notions (e.g. the family of c.c.c. ones) meaning: for at
least one of them. 

\noindent
5) We replace unstable by stable for the negation. 
\end{definition}

\begin{observation}
\label{7.11}
If $\Delta$ is closed under negation,
\then \, ${\gA}$ is weakly $(\aleph_0,\Delta,\lambda)$-unstable
iff ${\gA}$ is $(\aleph_0,\Delta,\lambda)$-unstable.
\end{observation}
\newpage

\section {Rank and Indiscernibility}

\begin{definition}
\label{7.12}
Let $({\gA},\Delta)$ be a $\kappa$-candidate.  Let $1 \le m \in \bbN$
and assume $D$ is a $\kappa^+$-complete filter on ${}^m \gA$; (or a
filter on some $\mathbf I \subseteq {}^m \gA$, then we interpret it as
$\{\mathbf J \subseteq {}^m \gA:\mathbf J \cap \mathbf I \in D\}$; writing
$\rk^m_\lambda$ we mean $\rk^m_D$ means $D = \{\mathbf I \subseteq {}^m
\gA:{}^m \gA \backslash \mathbf I$ has cardinality $\le \lambda\}$;
similarly $\rk^n_{< \lambda}$).

For $\mathbf I \subseteq {}^m{\gA},\mathbf I \ne \emptyset \mod D$ 
we define $\rk^m_D(\mathbf I) =
\rk^m_D(\mathbf I,\Delta,{\gA})$, an ordinal or infinity or $-1$
by defining for any ordinal $\alpha$ when $\rk^m_D(\mathbf I) \ge \alpha$
by induction on $\alpha$.
\end{definition}
\bigskip

\noindent
\underline{Case 1}:  $\alpha = 0$.

$\rk^m_D(\mathbf I) \ge \alpha$ iff $\mathbf I \ne \emptyset \mod D$.
\bigskip

\noindent
\underline{Case 2}:  $\alpha$ limit.

$\rk^m_D(\mathbf I) \ge \alpha$ iff $\rk^m_D(\mathbf I) \ge \beta$ for every
$\beta < \alpha$.
\bigskip

\noindent
\underline{Case 3}:  $\alpha = \beta +1$.

$\rk^m_D(\mathbf I) \ge \alpha$ iff (a) + (b) holds where
\mn
\begin{enumerate}
\item[$(a)$]   if $\mathbf I = \cup\{\mathbf I_i:i < \kappa\}$ then for
some $i < \kappa$ we have $\rk^m_D(\mathbf I_i) \ge \beta$
\sn
\item[$(b)$]   we can find $\bar \varphi(\bar x,\bar b) \in
\Phi^m_{{\gA},\Delta}$ and $\mathbf I_0,\mathbf I_1 \subseteq \mathbf I$ such that 
$\rk^m_D(\mathbf I_\ell) \ge \beta$ and
$\bar a \in \mathbf I_\ell \Rightarrow {\gA} \models \varphi_\ell(\bar a,\bar
b)$ for $\ell = 0,1$.
\end{enumerate}

\begin{observation}
\label{7.13}
Assume $({\gA},\Delta)$ is a $\kappa$-candidate.

\noindent
1) If $\alpha \le \beta$ are ordinals and $\rk^m_D(\mathbf I) \ge \beta$
\then \, $\rk^m_D(\mathbf I) \ge \alpha$. 

\noindent
2) $\rk^m_D(\mathbf I) \in \Ord \cup \{-1,\infty\}$ is well
defined (for $\mathbf I \subseteq {}^m {\gA})$. 

\noindent
3) If $\mathbf I_1 \subseteq \mathbf I_2 \subseteq {\gA}$ then $\rk^m_D(\mathbf I_1)
\le \rk^m_D(\mathbf I_2)$.
\end{observation}

\begin{PROOF}{\ref{7.13}}
Trivial.
\end{PROOF}

\begin{theorem}
\label{7.14}
The following are equivalent if
$2^{\aleph_0} \ge \kappa^+$ and $({\gA},\Delta)$ is a $\kappa$-candidate:
\mn
\begin{enumerate}
\item[$(a)$]   $\rk^m_\kappa({}^m{\gA}) \ge \kappa^+$ for some $m$
\sn
\item[$(b)$]   ${\gA}$ is $(\aleph_0,\Delta)$-unstable, see Definition
\ref{7.1}(3) 
\sn
\item[$(c)$]   ${\gA}$ is $(\aleph_0,\Delta,\kappa^+)$-unstable, see
  Definition \ref{7.10}(1)
\sn
\item[$(d)$]   $\rk^m_\kappa({}^m{\gA}) = \infty$ for some $m$
\sn
\item[$(e)$]   $\gA$ is not $(\aleph_0,\Delta)$-stable, see Definition
\ref{7.1}(2)
\sn
\item[$(f)$]   the assumption $\boxplus$ of \ref{7.7A}.
\end{enumerate}
\end{theorem}
 
\begin{PROOF}{\ref{7.14}}
We first prove $(a) \Rightarrow (b) \Rightarrow (c) \Rightarrow (d)
\Rightarrow (a)$, then $(c) \Rightarrow (e), \neg(d) \Rightarrow
\neg(e)$ and lastly $(a) \Rightarrow (f) \Rightarrow (b)$; clearly
this is sufficient.

\noindent
\underline{$(a) \Rightarrow (b)$}.

Let ${\cP}_\alpha = \{\mathbf I \subseteq {}^m{\gA}:\rk^m_\kappa(\mathbf I) 
\ge \alpha\}$ for $\alpha < \kappa^+$ and apply Claim \ref{7.7A}.
\medskip

\noindent
\underline{$(b) \Rightarrow (c)$}:  Trivial recalling Definition
\ref{7.1}, \ref{7.10} as we are assuming $2^{\aleph_0} \ge \kappa^+$.
\medskip

\noindent
\underline{$(c) \Rightarrow (d)$}:

Let $A \subseteq {\gA}$ be countable and $\{\bar a_\alpha:\alpha <
\kappa^+\} \subseteq {}^m{\gA}$ exemplify that ${\gA}$ 
is $(\aleph_0,\Delta,\kappa^+)$-unstable. 

Now \wilog \,
\mn
\begin{enumerate}
\item[$(*)$]   if $\bar b \subseteq A,\varphi(\bar x,\bar y) \in
  \Delta$ then 
\sn
\begin{enumerate}
\item[$\bullet$]  $\{\alpha < \kappa^+:{\gA} \models
\varphi(\bar a_\alpha,\bar b)\}$ is either unbounded or is empty
\sn
\item[$\bullet$]  $\{\alpha < \kappa^+:\gA \models 
\neg\varphi[\bar a_\alpha,\bar b]\}$ is either unbounded or is empty.
\end{enumerate}
\end{enumerate}
\mn
Now let ${\cP} = \{\{\bar a_\alpha:\alpha \in S\}:S \subseteq
\kappa^+$ is unbounded$\}$.  Now we can prove by
induction on $\alpha$ that $\mathbf I \in {\cP} \Rightarrow
\rk^m_D(\mathbf I) \ge \alpha$ recalling \ref{7.10}(1)(b).
\medskip

\noindent
\underline{$(d) \Rightarrow (a)$}:  Trivial.
\medskip

\noindent
\underline{$\neg(e) \Rightarrow \neg(c)$}:

Let $m,M$ and $\langle \bar a_\alpha:\alpha < \kappa^+\rangle$ be as in
Definition \ref{7.10}(1).  Let
$\chi$ be large enough and let $N \prec (\cH(\chi),\in)$ be of
cardinality $\kappa$ such that $\kappa + 1 \subseteq N$ hence $M
\subseteq N$ and $\gA,\Delta,M$ and $\langle \bar a_\alpha:\alpha <
\kappa^+\rangle$ belong to $N$.  So by \ref{7.2}(3) for every $\alpha
< \kappa^+$ for some $\mathbf c \in N$ we have $(**)$ of $\boxplus$ of
Definition \ref{7.1}(2).  As $\|N\| = \kappa$ necessarily for some
$\mathbf c_* \in N$ we have $S = \{\alpha < \kappa^+:\mathbf c_\alpha =
\mathbf c_*\}$ has cardinality $\kappa^+$.  Choose $\alpha \ne \beta$
from $S$ and get a contradiction to the choice of the $\langle \bar
a_\alpha:\alpha < \kappa^+\rangle$.
\medskip

\noindent
\underline{$\neg(d) \Rightarrow \neg(e)$}: 

We have to prove $\neg(e)$, i.e. $\gA$ is $(\Delta,\aleph_0)$-stable,
see Definition \ref{7.1}(2).  We use \ref{7.14}(c).

So we can find $m < \omega,N \prec (\cH(\chi),\in),\gA \in N,\kappa +1
\subseteq N$ and $\|N\| = \kappa$ and $\bar a \in {}^m \gA$ and it
suffices to find a function $\mathbf c$ in \ref{7.1}(2).  Now let $\cZ :=
\{\rk^m_\kappa(\mathbf I):\mathbf I \in \cB\}$ where $\cB = \{\mathbf I:\mathbf I
\in N$ is a non-empty subset of
${}^m \gA$ to which $\bar a$ belongs$\}$.  Now this family $\cB$ is not
empty (as ${}^m \gA \in \cB$) hence $\cZ \ne \emptyset$.

Of course, $\infty \notin \cZ$ by our present
assumption, i.e. $\neg(d)$.  Hence $\beta = \min(\cZ)$ is a well
defined ordinal, so by the definition of $\cZ$ there is $\mathbf I \in
\cB$ such that $\rk^m_\kappa(\mathbf I) = \beta$; clearly $\mathbf I \in N$
by the choice of $\cB$, hence also $\beta \in N$.  Now why
$\rk^m_\kappa(\mathbf I) \ngeq \alpha := \beta +1$?  By Definition
\ref{7.12}, case 3, clause (a) or clause (b) there fail.
\medskip

\noindent
\underline{Case 1}:  Clause (a) fails.

Then there is a sequence $\bar{\mathbf I} = \langle
\mathbf I_i:i < \kappa\rangle$ of subsets of $\mathbf I$ with union
$\mathbf I$ such that $i < \kappa \Rightarrow \rk^m_\kappa(\mathbf I_i) <
\beta$, clearly $\bar{\mathbf I} \in \cH(\chi)$ hence as $\mathbf I,\beta
\in N$ \wilog \,
$\bar{\mathbf I} \in N$.  As $\kappa \subseteq N$ necessarily $i < \kappa
\Rightarrow \mathbf I_i \in N$ and obviously $i < \kappa \Rightarrow \mathbf I_i
\subseteq \mathbf I \subseteq {}^m \gA$.

Lastly, as $\bar a \in \mathbf I$ necessarily there is $j < \kappa$ such
that $\bar a \in \mathbf I_j$.  So $\mathbf I_j$ witness
$\rk^m_\kappa(\mathbf I_j) \in \cZ$ and as said above
$\rk^m_\kappa(\mathbf I_j) < \beta$, so contradicting the choice of
$\beta$ as $\min(\cZ)$.
\medskip

\noindent
\underline{Case 2}:  Clause (b) (of case 3 of Definition \ref{7.12})
fails (for our $\alpha,\beta,\mathbf I$).

So
\mn
\begin{enumerate}
\item[$\odot_1$]  for every $\bar\varphi(\bar x,\bar b) \in
\Phi^m_{\gA,\Delta}$ we can choose $\iota < 2$ such that
$\rk^m_\kappa(\{\bar a' \in \mathbf I:\gA \models \varphi_\iota[\bar a',\bar
b]\}) < \beta$.
\end{enumerate}
\mn
So there is a function $\mathbf t \in N$ with domain
$\Phi^m_{\gA,\Delta}$ and range $\subseteq \{0,1\}$ such that
\mn
\begin{enumerate}
\item[$\odot_2$]  if $\bar\varphi(\bar x,\bar b) \in
\Phi^m_{\gA,\Delta}$ then $\mathbf t = \mathbf t(\bar\varphi(\bar x,\bar
b))$ satisfies $\rk^m_\kappa(\mathbf I_{\bar\varphi(\bar x,\bar b),\mathbf t}) <
\beta$ where
\sn
\item[$\odot_3$]   $\mathbf I_{\bar\varphi(\bar x,\bar b),\mathbf t} = \{\bar a'
\in \mathbf I:\gA \models \varphi_{\mathbf t}[\bar a',\bar b]\}$.
\end{enumerate}
\mn
However
\mn
\begin{enumerate}
\item[$\odot_4$]  if $\bar\varphi(\bar x,\bar b) \in
\Phi^m_{\gA,\Delta} \cap N$ and $\mathbf t < 2$ and $\gA \models \varphi_{\mathbf
t}[\bar a,\bar b]$ then
\sn
\begin{enumerate}
\item[$(a)$]  $\mathbf I_{\bar\varphi(\bar x,\bar b),\mathbf t} \in \cB$
\sn
\item[$(b)$]  $\rk^m_\kappa(\mathbf I_{\bar\varphi(\bar x,\bar b),\mathbf
t}) \in \cZ$
\sn 
\item[$(c)$]  $\rk^m_\kappa(\mathbf I_{\bar\varphi(\bar x,\bar b),\mathbf
t}) = \beta$  
\sn
\item[$(d)$]  $\mathbf t \ne \mathbf t(\bar\varphi(\bar x,\bar b))$.
\end{enumerate}
\end{enumerate}
\mn
[Why?  For clause (a): as $\bar\varphi(\bar x,\bar b) \in N$ and
$\mathbf I \in \cB \subseteq N$ clearly 
$\mathbf I_{\bar\varphi(\bar x,\bar b),\mathbf t} \in N$ and it is
$\subseteq \mathbf I \subseteq {}^m \gA$.  Also, $\bar a \in
\mathbf I_{\bar\varphi(\bar x,\bar b),\mathbf t}$ by the assumption of
$\odot_4$, together $\mathbf I_{\bar\varphi(\bar x,\bar b),\mathbf t} \in
\cB$ as required.  Hence by the definition of $\cZ$ also clause (b) holds.
But by \ref{7.13}(3) we have $\rk^m_\kappa(\mathbf I_{\bar\varphi(\bar x,
\bar b),\mathbf t}) \le \rk^m_\kappa(\mathbf I) = \beta$ so by the choice
of $\beta$ as $\min(\cZ)$ we get equality, i.e. clause (c) holds.
Recalling $\odot_2$ we get clause (d).]

Now by $\odot_4$ (as $\mathbf t$ is a function from $N$), we are done
proving $\neg(e)$.
\medskip

\noindent
\underline{$(a) \Rightarrow (f)$}:  

For $\alpha < \kappa^+$ let
$\cP_\alpha = \{\mathbf I \subseteq {}^m \gA:\rk^m_K(\mathbf I) \ge
\alpha\}$ and recalling Definition \ref{7.12} it is easy.
\medskip

\noindent
\underline{$(f) \Rightarrow (b)$}:  By \ref{7.7A}.
\end{PROOF}

\begin{conclusion}
\label{7.14d}
1) The property ``$\rk^m_\kappa(\gA) = \infty$" is preserved by forcing.

\noindent
2) If $\rk^m_\kappa({}^m \gA) = \infty$ 
\then \, for some $\cU_1 \subseteq \kappa$
   coding the definitions of $\gA$ and of the $\kappa$-Suslin
   definitions of every $\varphi(\bar x) \in \Delta$, and $\cU_2
\subseteq \Ord$ every forcing extension of $\mathbf L[\cU_1,\cU_2]$
satisfies $\rk^m_\kappa({}^m \gA) = \infty$.
\end{conclusion}

\begin{PROOF}{\ref{7.14d}}
1) By Theorem \ref{7.14} it suffices to prove clause (b) there
   ($(\gA,\Delta)$ is $(\aleph_0,\Delta)$-unstable) is preserved, but
   by Theorem \ref{7.14}, i.e. $(a) \Rightarrow (f)$ there, the
   assumption of \ref{7.7A} holds but its conclusion says that
``$(\gA,\Delta)$ is $(\aleph_0,\Delta)$-unstable" holds also in
   forcing extensions of $\mathbf V$.

\noindent
2) Similarly.
\end{PROOF}

\begin{definition}
\label{7.15}
If $p$ is a $(\Delta_1,m)$-type 
over $A$ in ${\gA}$ (i.e. a set of formulas $\varphi(\bar x,\bar a)$ 
with $\varphi(\bar x,\bar y) \in \Delta_1,\bar a \subseteq A$), we let
(may write $\mu$ instead of $< \mu^+$)

\begin{equation*}
\begin{array}{clcr}
\rk^m_{< \lambda}(p,\Delta_1,{\gA}) = 
\Min\{\rk^m_{< \lambda} (\bigwedge\limits_{\ell < n} 
\varphi_\ell({}^m{\gA},\bar b_\ell),\Delta_1,{\gA}):&n < \omega \\
  &\text{and } \varphi_\ell(\bar x,\bar b_\ell) \in p \text{ for }
\ell < n\}.
\end{array}
\end{equation*}
\end{definition}

\begin{observation}
\label{7.16}
1) If $p \subseteq q$ (or just $q \vdash
p$) are $(\Delta,m)$-types in ${\gA}$ \then \,
$\rk^m_{< \lambda}(q,\Delta,{\gA}) \le \rk^m_{< \lambda}(p,\Delta,A)$. 

\noindent
2) If $q$ is a $(\Delta,m)$-type in ${\gA}$ \then \, for some finite
$p \subseteq q$ we have

\[
\rk^m_{<\lambda}(q,\Delta,{\gA}) = \rk^m_{<\lambda}(p,\Delta,{\gA})
\]

\mn
hence

\[
p \subseteq r \subseteq q \Rightarrow \rk^m_{<\lambda}(r,\Delta,{\gA}) =
\rk^m_{<\lambda}(p,\Delta,{\gA}).
\]
\end{observation}

\begin{PROOF}{\ref{7.16}}
Obvious by the definitions.
\end{PROOF}

\begin{claim}
\label{7.17}
Assume $\kappa^+ \le 2^{\aleph_0}$.   In \ref{7.14} we can add
\mn
\begin{enumerate}
\item[$(g)$]   for some $\mu < \cf(\lambda)$ and, of course, $\kappa <
\cf(\lambda),\lambda \le 2^{\aleph_0}$ the pair $({\gA},\Delta)$ is
$(\mu,\Delta,\lambda)$-unstable
\sn
\item[$(h)$]   like (g) for every such $\mu,\lambda$
\sn
\item[$(i)$]   $({\gA},\Delta)$ is not $\aleph_0$-stable.
\end{enumerate}
\end{claim}

\begin{PROOF}{\ref{7.17}}

\noindent
\underline{$\neg(i) \Rightarrow \neg(c)$}.

Let $M \prec ({\cH}(\chi),\in,<^*_\chi)$ be countable such that $(\gA,\Delta)
\in M$ and $m < \omega$.  For every $\bar a \in
{}^m{\gA}$ there is a function $\mathbf c_{\bar a} \in M$ from
$\Phi^m_{({\gA},\Delta)}$ to $\{0,1\}$ as in Definition
\ref{7.1}.  So if $\bar a_i \in {}^m{\gA}$ for $i < \kappa^+$
then for some $i < j < \kappa^+$ we have 
$\mathbf c_{\bar a_i} = \mathbf c_{\bar a_j}$ because $M$ is countable
hence for no $\bar\varphi(\bar x,\bar b) \in \Phi^m_{(M,\Delta)}$ do we
have $\varphi_0(\bar x,\bar b) \in \tp_\Delta(\bar a_i,
M,\gA),\varphi_1(\bar x,\bar b) \in \tp_\Delta(\bar a_j,M,\gA)$.
So clearly $(c)$ fails.
\medskip

\noindent
\underline{$(i) \Rightarrow (c)$}.

Fix $({\cH}(\chi_0),\in,<^*_\chi)$ and let

\begin{equation*}
\begin{array}{clcr}
\mathscr{S}_0 = \{M \prec ({\cH}(\chi_0),\in,<^*_\chi): \,&{\gA} \in M
\text{ and } \|M\| = \kappa \\
  &\text{and } \kappa +1 \subseteq M\}.
\end{array}
\end{equation*}

\mn
For $m < \omega$ and $\mathbf I \subseteq {}^m{\gA}$ let 
${\cJ}_{\mathbf I} = {\cJ}[\mathbf I]$ be the family of $\mathscr{S} \subseteq
\mathscr{S}_0$ such that: we can find $\langle F_x,\mathbf c_x:x \in 
{\cH}(\chi) \rangle$ (a witness) such that:
\mn
\begin{enumerate}
\item[$(\alpha)$]   $\mathbf c_x:\Phi^m_{{\gA},\Delta} \rightarrow \{0,1\}$
\sn
\item[$(\beta)$]   $F_x:{}^{\omega >}({\cH}(\chi)) \rightarrow
{\cH}(\chi)$
\sn
\item[$(\gamma)$]   if $M \in \cS_0$ is closed under $F_x$ for $x \in M$
then for every $\bar a \in \mathbf I$ for some $y \in M,\mathbf c_y$ is 
a witness for tp$(\bar a_M,M \cap {\gA},{\gA})$, see Definition \ref{7.1}(2).
\end{enumerate}
\mn
Clearly ${\cJ}_{\mathbf I}$ is a normal ideal on $\mathscr{S}_0$.  Also
if ``$m < \omega \Rightarrow \mathscr{S}_0 \in {\cJ}[{}^m{\gA}]$" then
increasing $\chi$ we get the desired result.  Toward contradiction
assume that $m < \omega$ and $\mathscr{S}_0 \notin {\cJ}[{}^m{\gA}]$
and let ${\cP}$ (i.e. ${\cP}_\alpha = {\cP}$ for $\alpha <
\kappa^+$) be the family of $\mathbf I \subseteq {}^m{\gA}$
such that $\mathscr{S}_0 \notin {\cJ}_{\mathbf I}$.

We now finish by \ref{7.7A} once we prove
\mn
\begin{enumerate}
\item[$\circledast$]   if $\mathbf I \in {\cP}$ then for some $\bar
\varphi(\bar x,\bar b) \in \Phi^m_{{\gA},\Delta}$ for each $\ell <
2$ the set $\mathbf I^\ell_{\bar \varphi(x,\bar b)}$ is $\{\bar a \in
\mathbf I:{\gA} \models \varphi_\ell(\bar a,\bar b)\}$ belong to
${\cP}$.
\end{enumerate}
\mn
If not, for every $\bar \varphi(\bar x,\bar b) \in
\Phi^m_{{\gA},\Delta}$ there is $\ell = \mathbf c[\bar \varphi(\bar
x,\bar b)] < 2$ and $\langle (F^{\bar \varphi(\bar x,\bar b)}_x,\mathbf
c^{\bar \varphi(\bar x,\bar b)}_x):x \in {\cH}(\chi) \rangle$
witnessing $\mathscr{S}_0 \in \cJ[\mathbf I^\ell_{\bar \varphi}]$.

Define $(F_y,\mathbf c_y)$ for $y \in {\cH}(\chi)$ by: if $y =
\langle x,\bar \varphi(\bar x,\bar b) \rangle$ then $F_y = F^{\bar
\varphi(\bar x,\bar b)}_x,c_y = \mathbf c^{\bar \varphi(\bar x,b)}_x$,
otherwise $\mathbf c$.

Clearly we can find $M \in \mathscr{S}_0$ such that
\mn
\begin{enumerate}
\item[$\circledast_1$]   if $\bar\varphi(\bar x,\bar b) \in
\Phi^m_{{\gA},\Delta} \cap M$ and $x \in M$ then $M$ is closed
under $F^{\bar \varphi(\bar x,\bar b)}_x$
\sn
\item[$\circledast_2$]    for some $\bar a \in {}^m{\gA}$, no
$\mathbf c_y,y \in M$ defines tp$_\Delta(\bar a,M \cap {\gA},{\gA})$, in
the sense of \ref{7.1}(2).
\end{enumerate}
\mn
But $\mathbf c$ does it!  So we are done.  Alternatively, note $(i) =
(c)$ so we can quote \ref{7.14}.
\medskip

\noindent
\underline{$(h) \Rightarrow (g)$}.

Obvious.
\medskip

\noindent
\underline{$(g) \Rightarrow (d)$}.

Like $(c) \Rightarrow (d)$.
\medskip

\noindent
\underline{$(c) \Rightarrow (h)$}.

Repeat the proof of ``$(c) \Rightarrow (d)$" in \ref{7.14}.
\end{PROOF}

\begin{theorem}
\label{7.18}
Assume that $({\gA},\Delta)$ is a
$\kappa$-candidate and is $\aleph_0$-stable or just $\mu$-stable.

For some $\xi < \kappa^+$ we have: if $\lambda \ge \mu,m < \omega,A
\subseteq {\gA},|A| \le \lambda$ and $\bar a_\alpha \in {}^m{\gA}$ 
for $\alpha < \lambda^{+ \xi}$ \then \, for some $S \subseteq
\lambda^{+ \xi}$ of cardinality $\lambda$ the sequence $\langle \bar
a_\alpha:\alpha \in S \rangle$ is a $\Delta$-indiscernible sequence over $A$ in
${\gA}$.
\end{theorem}

\begin{remark}
On getting indiscernible sets, see \ref{z20}, \ref{7.7}.
\end{remark}

\begin{PROOF}{\ref{7.18}}
First, fix $A,\lambda,\mu$ and assume the conclusion fails.  

For $\xi < \kappa^+$ let 

\begin{equation*}
\begin{array}{clcr}
{\cP}_\xi = \{\mathbf I:&\mathbf I \subseteq {}^m \gA \text{ has
  cardinality } \ge \lambda^{+ \xi} \text{ and no } \mathbf I' \subseteq
  \mathbf I \text{ of cardinality} \\
  &\lambda \text{ is } \mathbf I' \text{ a } \Delta 
\text{-indiscernible sequence over } A \text{ in } {\gA}\}.
\end{array}
\end{equation*}

\mn
So by Claim \ref{7.7A} it suffices to check the demands $\boxplus(a) -
(e)$ there are satisfied by $\bar{\cP} = \langle \cP_\xi:\xi <
\kappa^+\rangle$ defined above and $\Phi := \Phi^m_{(\gA,\Delta)}$.

Now clause (a) holds by by our assumptions (and the choice of
$\Phi$).  

Clause (b) holds by the choice of $\cP_\xi$ and clause (c) holds by
our assumption toward contradiction.  Also clause (e) there is trivial
because $\mathbf I \in \cP_\xi \wedge \mathbf I' \subseteq \mathbf I
\cap |\mathbf I'| \ge \kappa^{+ \zeta} \Rightarrow \mathbf I' \in
\cP_\zeta$.

So we are left with proving clause (d) there, for this (as $\cP_\xi$
 increases with $\xi$) it suffices to prove the following:
\mn
\begin{enumerate}
\item[$\circledast$]   if $\lambda^{+ \xi}$ is regular, $\xi > \zeta,
\bar a_\alpha \in {}^m {\gA}$
for $\alpha < \lambda^{+ \xi}$ and $S \subseteq \lambda^{+ \xi}$ is
stationary then (a) or (b) where:
\sn
\begin{enumerate}
\item[$(a)$]  for some club $E$ of $\lambda,\langle \bar
a_\alpha:\alpha \in S \cap E \rangle$ is a $\Delta$-indiscernible set over
$A$ in ${\gA}$
\sn
\item[$(b)$]  for some $\bar\varphi(\bar x,\bar b) \in \Phi^m(\gA,\Delta)$ the
sets $\mathbf I_0,\mathbf I_1$ have cardinality $\ge \lambda^{+\zeta}$
where $\mathbf I_\ell = \{\bar a \in \mathbf I:\gA \models 
\varphi_\ell[\bar a,\bar b]\}$ for $\ell=0,1$.
\end{enumerate}
\end{enumerate}
\mn
So we fix $\xi,\mathbf I$ and $\zeta$ as in $\circledast$ and we shall
prove that (a) of $\circledast$ or (b) of $\circledast$ holds.

Now
\mn
\begin{enumerate}
\item[$\oplus$]   for some $m < \omega$ and club $E^*_n$ of
$\lambda^{+ \xi}$ we have:
\sn
\begin{enumerate}
\item[$(i)$]   $\langle \bar a_\alpha:\alpha \in S \cap
E^*_n \rangle$ is $(\Delta,m)$-end extension indiscernible sequence
over $A$ in $\gA$
\sn
\item[$(ii)$]   for no club $E' \subseteq E^*_n$ of
$\lambda^{+ \xi}$ is $\langle \bar a_\alpha:\alpha \in S \cap E'
\rangle$ a sequence which is $(\Delta,m+1)$-end extension indiscernible.
\end{enumerate}
\end{enumerate}
\mn
[Why?  If for every $m$ there is a club $E^*_m$ of $\lambda^{+ \xi}$
  such that clause (i) of $\oplus$ holds, then $E = \cap\{E^*_n:n\}$
  is a club of $\lambda^{+ \xi}$ satisfying clause (i) of
  $\circledast$ for every $m$.

Hence $E$ is a club of $\lambda^{+ \xi}$ 
such that $\langle \bar a_\alpha:\alpha \in S \cap E\rangle$ is
$\Delta$-indiscernible sequence over $A$ in $\gA$, so as required in
clause (a) of $\circledast$, (hence is as required in the theorem), 
but we are assuming there is no such set, 
so for some $m$ there is no such $E^*_m$.  But for
$m=0$ there is such $E^*_m$: the club $E^*_0 := \lambda^{+\xi}$. Together
for some $m$ there is a club $E^*_m$ such that clause $(i)$ of
$\oplus$ holds for this $m$ 
but there is no such club for $m+1$.  So $\oplus$ holds indeed.]

By claim \ref{7.4}(3) there is a club $E$ of $\lambda$ and 
$\langle f_k:k < \omega \rangle$ as there, (in fact, only $f_m,f_{m+1}$
are used), let $S^*_\gamma = \{\alpha \in S:\alpha > \gamma$ and
$f_{m+1}(\alpha) = \gamma\}$, and let $W_{m+1}$ be
$\{\gamma:S^*_\gamma$ is stationary$\}$. 

Clearly there is a $E_*$ such that:
\mn
\begin{enumerate}
\item[$(a)$]  $E_*$ is a club of $\lambda^{+ \xi}$
\sn
\item[$(b)$]  $E_* \subseteq E \cap E^*_m$
\sn
\item[$(c)$]  if $\gamma < \lambda^{+ \xi}$ but $\gamma \notin
  W_{n+1}$ then $S^*_\gamma$ is disjoint to $E_* \backslash (\gamma
  +1)$ 
and even to $E_*$
\sn
\item[$(d)$]  if $\sup(W_{n+1}) < \lambda^{+ \xi}$ then $\sup(W_{n+1})
  < \min(E_*)$
\end{enumerate}
\medskip

\noindent
\underline{Case 1}:  $W_{n+1} = \emptyset$

As $f_{n+1}$ is a pressing down function on $E$, a club of $\lambda^{+
  \xi}$ this is impossible.
\medskip

\noindent
\underline{Case 2}:  $W_{n+1}$ is a singleton, say $\{\gamma_*\}$

In this case $f_{m+1}$ is constant on $S \cap E_*$ hence clearly
$\langle \bar a_\alpha:\alpha \in S \cap E_*\rangle$ contradiction to
the choice of $m$.
\medskip

\noindent
\underline{Case 3}:  $W_n$ has at least two elements

Let $\gamma_1 < \gamma_2$ belongs to $W_{n+1}$, applying \ref{b32} we
get the desired result.

Well, we still have a debt: why $\xi$ does not depend on
$(A,\mu,\lambda)$?  So assume for each
$\xi,(A_\xi,\mu_\xi,\lambda_\xi,\xi,\cI_\xi)$ forms a counterexample.
Now formalizing the above and we quote ``$\psi$ has a model such that
$(Q^M,<^M)$ is not well ordered when $\psi \in
\bbL_{\kappa^+,\aleph_0}$ has a model $M_\xi,(Q^{M_\xi},<^{M_\xi}) =
(\xi,\alpha)$ for arbitrarily large $\xi < \kappa^+"$, see Keisler \cite{Ke71}.
\end{PROOF}

\begin{claim}
\label{c19}
Assume $(\gA,\Delta)$ is a $\kappa$-candidate and it is
$(\aleph_0,\Delta)$-unstable, see Claim \ref{7.14} and Definition
\ref{7.1}(3).
\Then \, $\gA$ is not categorical$_1$, even ``under $2^{\aleph_0} =
\aleph_1$", see \ref{z7f}(3A).
\end{claim}

\begin{remark}
Of course, we can get stronger versions: many models.
\end{remark}

\begin{PROOF}{\ref{c19}}
Let $\bar x = \bar x_m,\bar\varphi_\nu = (\varphi_{\nu,0}(\bar
x,y_\nu),\varphi_{\nu,1}(\bar x,\bar y)) \in \Phi^m_{(\gA,\Delta)}$
for $\nu \in {}^{\omega >}2$ and $\langle \bar b_\nu:\nu \in
{}^{\omega >}2\rangle,\langle \bar a_\eta:\eta \in {}^\omega 2\rangle$
be as in Definition \ref{7.1}(3).

\Wilog \, this holds absolutely, i.e. if $\bbP$ is a forcing extension,
and $\eta \in ({}^\omega 2)^{\mathbf V[\bbP]}$ then we can choose $\bar
a_\eta$.

Let $\bbQ$ be the forcing of adding $\aleph_1$ Cohens,
$\name{\bar\eta} = \langle \name \eta_\alpha:\alpha <
\aleph_1\rangle$ the generic and $\bbP$ be trivial and easily 
$\gA = \gA^{\mathbf V},\gA^{\mathbf V[\bbQ]}$ are not isomorphic: let
$\name F$ be a $\bbQ$-name and assume toward contradiction $p \Vdash ``\name F$
is an isomorphism from $\gA^{\mathbf V[\bbQ]}$ onto $\gA^{\mathbf V}"$
where $p \in \mathbf G$ and $\mathbf G \subseteq \bbQ$ is generic over
$\mathbf V$.  So for some $\alpha(*) < \aleph_1,\langle
\name F \rest (\name b_\nu):\nu \in {}^{\omega >}2\rangle$ depend just
on $\langle \name \eta_\alpha:\alpha < \alpha(*)\rangle$ so in $\mathbf
V[\name{\bar\eta} \rest \alpha(*)]$ we can compute it, so $\name F(\name
\eta_{\alpha(*)})$ can have no possible value, contradiction.
\end{PROOF}
\newpage

\section {Concluding Remarks}

We try to expand on justifying \ref{0.8}, that is the 
connection to categoricity of $\psi \in \bbL_{\aleph_1,\aleph_0}$ and
how much ``categoricity in $\lambda$" depends on $\lambda$.

First we may consider more relatives of \ref{0.10}.

\begin{definition}
For a definition $\gA$ of a $\tau$-model (usually with a set of
elements a definable set of reals) and $\lambda =
\lambda^{\aleph_0}$; we say $\gA$ is $\gK$-categorical$_3$ in
$\lambda$ \when \,:  if $\bbP_2 = \bbP_0 \times \bbP_1$ are forcing
notions from $\gK$ such that $|\bbP_\ell| = \lambda$ and
$\Vdash_{\bbP_\ell} ``2^{\aleph_0} = \lambda"$, \then \, there is a
forcing notion $\bbP_3 \in \gK$ such that $\bbP_2 \lessdot \bbP_3$,
moreover $\bbP_2 \le_{\gK} \bbP_3$ and
$\Vdash_{\bbP_3} ``\gA^{\mathbf V[\bbP_0]},\gA^{\mathbf V[\bbP_1]}$ are
isomorphic". 

It may be helpful to first analyze (Cohen,$\lambda$)-categoricity$_2$,
assuming $0^{\#}$ exists, so $\gK = \{\bbQ$: for some cardinal $\mu,\bbQ$ is
the forcing notion of adding $\mu$ Cohen reals$\}$ order by $\bbP
\le_{\gK} \bbQ$ iff $\bbP \lessdot \bbQ$ and $\bbP,\bbP/\bbQ$ are from
$\gK$.  Hence $\bbP_\ell$ is
adding $\lambda$ Cohen reals for $\ell=1,2$ and $\bbP_3$ is (isomorphic
to) adding $(\lambda + \lambda)$-Cohen reals. 
So let $\bar\eta_1 \char 94 \bar\eta_2$ be generic
for adding $(\lambda + \lambda)$-Cohens to $\mathbf V$ and $\bar\eta_\ell =
\langle \eta^\ell_\alpha:\alpha < \lambda\rangle$ for $\ell =1,2$.  (So
$\bar\eta_\ell$ is generic over $\mathbf L$ and for notational
simplicity we may assume that the $\mathbf V = \mathbf L$ so $\tau$ and the
definition of $\gA$ belongs to $\mathbf L$).

For $\ell=1,2$ and $u \subseteq \lambda$ let $\gA_{\ell,u}$
be $\gA^{\mathbf V[\name{\bar\eta}_\ell \rest u]}$ and the forcing notions
$\bbP_{\ell,u} \lessdot \bbP_\ell$ for $\ell=1,2,3$ be naturally
defined.  From the categoricity assumption it follows that \wilog \,
$\bbP_3$ is adding $(\lambda + \lambda + \lambda)$-Cohen reals
$\bbP_\ell = \bbP_{3,[\lambda \ell,\lambda \ell + \lambda]}$, and
there is a $\bbP_3$-name $\name f$ such that
\mn
\begin{enumerate}
\item[$(*)$]  $\Vdash_{\bbP_3} ``\name f$ is an isomorphism from
 $\gA[\name{\bar\eta}_1]$ onto $\gA[\name{\bar\eta}_2]"$.
\end{enumerate}
\mn
So in $\mathbf V$, for some function $F:{}^\omega \lambda \rightarrow
\lambda$ we have
\mn
\begin{enumerate}
\item[$(*)$]  if $u \subseteq \lambda$ is infinite and closed under $F$
  then $f_n = \name f \rest \gA_{1,u}$ is a $\bbP_{1,u}$-name and is
an isomorphism from $\name{\gA}_{1,u}$ onto $\name{\gA}_{2,u}$.
\end{enumerate}
\mn
For $u \subseteq \lambda$ let $c \ell_F(u)$ be the closure of $u$ by
$F$.  Now if $\lambda \ge (2^{\aleph_0})^{+n}$ and $n \ge m$ \then \,
(by downward induction on $m$) we can find $\bar\alpha = \langle
\alpha_\ell:\ell \in [m,n)\rangle$ such that $k \in [m,n] \Rightarrow
  \alpha_k \notin c \ell_F(\{\alpha_\ell:\ell \in [m,n)$ but $\ell \ne
    k\} \cup (2^{\aleph_0})^{+m})$.  So letting for $v \subseteq n,u_v
    = c \ell_F(\{\alpha_k:k \in v\} \cup (2^{\aleph_0}))$ we have:
\mn
\begin{enumerate}
\item[$(*)$]   $(a) \quad f_{u_v}$ is an isomorphism from $\gA_{1,u_v}$ 
onto $\gA_{2,u_v}$ increasing with $v$ 

\hskip25pt  for $v \subseteq n$
\sn
\item[${{}}$]   $(b) \quad$ if $f(1) \subseteq v(2) \subseteq n$ then
  $f_{u_{v(1)}} \subseteq f_{u_{v(2)}}$
\sn
\item[${{}}$]   $(c) \quad$ if $v \subseteq n,k<n$ and $k \notin v$
  then $\alpha_k \notin u_v$.
\end{enumerate}
\mn
This seems to suggest the stability of $(\mu,\cP^-(n))$-diagrams as in
\cite{Sh:87b}.  We expect that for every finite sequence $\bar a$ from
$\gA_{1,n}$ its type over $\cup\{\gA_{1,u}:u \in \cP^-(n)\}$ is definable
in a suitable sense.

Of course, we may well  have to consider also (assuming for
transparency $2^{\aleph_0} = \aleph_1$) in addition to the
$\gA_{1,u_v}$ of cardinality $\aleph_1$ also countable approximation
$\gA'_{1,u_v}$.

Also more seriously, maybe the stability of $(\cP^-(n),2^{\aleph_0})$-diagrams
will not be enough, because for $\lambda > (2^{\aleph_0})^{+ \omega}$
the $F$-closure of $u,|u| = (2^{\aleph_0})^{+ \omega}$ has bigger
cardinality.  In this case maybe this stops after $\omega_1$ steps.
\end{definition}
\newpage

 % PRIVATE PART 1 

\bibliographystyle{amsalpha}
\bibliography{shlhetal}

\end{document}